\documentclass[reqno, 11pt]{amsart}
\usepackage{lineno} 
\usepackage{etoolbox}

\usepackage{amsmath}
\usepackage{graphicx}
\usepackage{amsfonts}
\usepackage{amssymb}
\usepackage{color}
\usepackage{mathrsfs}
\usepackage{epsfig}
\usepackage{url}
\usepackage{mathrsfs,a4wide}
\usepackage{calligra}
\usepackage{amsmath}
\usepackage{amssymb}
\usepackage{graphicx}
\usepackage{amsthm}
\usepackage{thmtools}
\usepackage{enumitem}
\usepackage{hyperref, cleveref}
\usepackage{mathrsfs}
\usepackage{tikz}

\theoremstyle{plain}

\theoremstyle{definition}
\theoremstyle{remark}
\numberwithin{equation}{section}

\theoremstyle{plain} \declaretheorem[numberwithin = section, name = Theorem,
 refname = {Theorem}, Refname = {Theorem}]{thm}
\theoremstyle{plain} \declaretheorem[numberlike = thm, name = Proposition,
 refname = {Proposition}, Refname = {Proposition}]{prop}
\theoremstyle{plain} \declaretheorem[numberlike = thm, name = Lemma,
refname = {Lemma}, Refname = {Lemma}]{lem}
\theoremstyle{plain} \declaretheorem[numberlike = thm, name = Definition,
 refname = {Definition}, Refname = {Definition}]{df}
\theoremstyle{definition} 
\theoremstyle{definition} \declaretheorem[numberlike = thm, name = Remark,
refname = {Remark}, Refname = {Remark}]{rem}
\theoremstyle{plain} \declaretheorem[numberlike = thm, name = Corollary,
refname = {Corollary}, Refname = {Corollary}]{cor}

\DeclareMathOperator {\R}{\mathbb{R}}
\DeclareMathOperator {\BV}{BV}
\DeclareMathOperator{\weak*}{\begin{picture}(10,4)
							\put(0,-2){$\rightharpoonup$}
							\put(3,3){$\ast$}
							\end{picture}}
\DeclareMathOperator {\h}{\mathcal{H}}
\DeclareMathOperator{\gammalimsup}{\Gamma\text{-}\limsup}

\DeclareMathOperator{\conv}{\text{conv}}

\title[Nonlocal phase transitions]
{$\Gamma$-convergence for nonlocal phase transitions involving the $H^{1/2}$ norm and surfactants}

\author[G. Fusco]{Giuliana Fusco}
\address[Giuliana Fusco]{Scuola Superiore Meridionale, via Mezzocannone 4, 80134 Napoli, Italy}
\email[]{g.fusco@ssmeridionale.it} 

\author[T. Heilmann]{Tim Heilmann}
\address[Tim Heilmann]{Technische Universit\"at M\"unchen, Boltzmannstrasse 3, 85748 Garching, Germany	}
\email[]{heilmant@ma.tum.de}

\begin{document}
	
\begin{abstract}
We study functionals
\begin{align*}
	F_\varepsilon (u,\rho) &:= \frac{1}{\varepsilon} \int_\Omega W(u)  \, dx
		+ \frac{1}{|\ln(\varepsilon)|} \int_\Omega  \int_\Omega 
		\frac{(u(y) - u(x))^2}{|y - x|^{N+1}} \, dy \,dx \\
	&\quad\quad + \frac{1}{|\ln(\varepsilon)|} \int_\Omega  \left| \int_\Omega 
		\frac{(u(y) - u(x))^2}{|y - x|^{N+1}} \, dy - 
		\rho(x) \right| \,dx
\end{align*}
for a double-well potential $W$ and a nonlocal, critically scaled gradient-like term,
together with a surfactant term. We show compactness in the space of $BV$ functions on $\Omega$
and the $\Gamma$-convergence to an energy  given as local perimeter-type functional, depending
also on the limit density of surfactant on the interface, plus the total variation of the surfactant measure
away from the interface.
\end{abstract}

\maketitle

\section{Introduction}
We consider a model for phase transitions with nonlocal interactions and with surfactants,
which is a modification of the classical van der Waals-Cahn-Hillard energy. The latter is a phase-field
approximation of the perimeter functional and it has been analyzed by Modica and Mortola in terms
of $\Gamma$-convergence in the limit as the transition length vanishes in \cite{m, mm}. 
We study a modified version of the  Cahn-Hilliard functional replacing the gradient term with a suitable rescaling
of the $\frac{1}{2}$-Gagliardo seminorm and, in the spirit of \cite{fms},
an additional integral term accounting for a fluid-surfactant interaction to model the adsorption
of surface active molecules onto phase interfaces.
We obtain that the effective surface
tension energy of a phase transition decreases as the density of the surfactant at the phase interface
increases and stays below a given constant $k$, while excess surfactant increases the limit energy.
 We point out that the variational analysis of energies involving an extra surfactant term has been recently studied also in \cite{ch1, ch2}.
 
Given a regular bounded open set $\Omega \subset \R^N$ (the region occupied by the fluid and the
 surfactant), one considers a scalar function $u:\Omega \to \R$ and a non-negative function 
$\rho:\Omega\to [0,+\infty)$ representing the order parameter of the fluid and the density of the surfactant, respectively.
For $u,\rho \in L^1(\Omega)$ we set $\mu=\rho\,dx$ and we define for $\varepsilon>0$ the family of energy functionals
\begin{align*}
	F_\varepsilon(u,\mu) :=&  \frac{1}{\varepsilon} \int_\Omega W(u(x)) dx 
		+ \frac{1}{|\ln(\varepsilon)|} \int_\Omega  \int_\Omega 
		\frac{(u(y) - u(x))^2}{|y - x|^{N+1}} \, dy \,dx \\
	&\quad\quad + \frac{1}{|\ln(\varepsilon)|} \int_\Omega  \left| \int_\Omega 
		\frac{(u(y) - u(x))^2}{|y - x|^{N+1}} \, dy - 
		\rho(x) \right| \,dx \,.
\end{align*}
Assuming $W:\R \rightarrow [0,+\infty)$ to be a double-well potential with wells at $\alpha$ and $\beta$,
we are interested  in the asymptotic behaviour as $\varepsilon$ tends to zero of $F_\varepsilon(u,\mu)$
in the sense of $\Gamma$-convergence (see \cite{b, dm}). The first two integral terms in $F_\varepsilon(u,\mu)$ define the
functional $NL_\varepsilon$, a non-local version of the classical Cahn-Hillard functional, namely
\begin{equation}
\label{NL_DEFINITION}
	NL_\varepsilon(u) :=  \frac{1}{\varepsilon} \int_\Omega W(u(x)) dx 
		+ \frac{1}{|\ln(\varepsilon)|} \int_\Omega  \int_\Omega 
		\frac{(u(y) - u(x))^2}{|y - x|^{N+1}} \, dy \,dx \,.
\end{equation}
The variational analysis for $NL_\varepsilon$ as $\varepsilon$ vanishes has been carried out in
the case for one space dimension $N=1$ in \cite{abs} and in arbitrary space dimension in
 \cite{sv1} using the results from \cite{sv2} and \cite{psv}, and with a different proof in \cite{h}.
In particular it has been proved the pre-compactness in $BV(\Omega;\{\alpha,\beta\})$ of sequences of
phase-fields $u_\varepsilon$ with uniformly bounded $NL_\varepsilon$ energy
and it has been computed the $\Gamma$-limit of $NL_\varepsilon$ as
$\varepsilon$ tends to zero with respect to the $L^1$ convergence.
Roughly speaking, the limit $u$ of a converging subsequence of energy bounded 
$u_\varepsilon$ is forced to take the values $\alpha$ and $\beta$ almost everywhere,
partitioning $\Omega$ in the two sets $\{u=\alpha\}$ and $\{u=\beta\}$.
These two sets are interpreted as the pure phases of the fluid while their
common boundary, corresponding to the jump set $S_u$ of the function $u$, as the phase interface. The
effective asymptotic energy of the system, captured by the $\Gamma$-limit of $NL_\varepsilon$, is proved
to be proportional to the surface measure of the interface $S_u$.
Such an energy is achieved along a sequence $u_\varepsilon$ of phase fields whose non-local
energy
$ 
\displaystyle\int_\Omega\frac{|u_\varepsilon(y) - u_\varepsilon(x)|^2}{|y-x|^{N+1}} \,dy$ 
concentrates on $x\in S_u$. In this perspective, one can understand the role of the third term in
$F_\varepsilon$ as modelling the interaction between the surfactant and the fluid and favouring the
phases of minimizers of $NL_\varepsilon$ to separate where the surfactant is present.

We would like to remark that nonlocal Modica-Mortola type energies of a different type
have been studied in \cite{ab,ab2}. There, the energy functionals have the form
\begin{equation*}
	\frac{1}{\varepsilon} \int_\Omega W(u)  \, dx +
		\varepsilon \int_{\Omega}\int_{\Omega} J_\varepsilon(y-x) \frac{(u(y) - u(x))^2}
		{\varepsilon^2} \, dy\, dx \,.
\end{equation*}
If the convolution  kernels $J_\varepsilon(\cdot) = \varepsilon^{-n} J(\cdot / \varepsilon)$
satisfy 
\begin{equation}
\label{kernel_bd_hp}
    \int_{\R^N} J(h) (|h| \wedge |h|^2) \,dh < \infty,
\end{equation} 
in \cite{ab} it is shown that,
as $\varepsilon \rightarrow 0^+$, the
limit energy functional is finite on the set of those BV functions on $\Omega$ that
attain only the values $\alpha, \beta$
and is given via a cell formula depending on $J$ resulting in a possibly anisotropic perimeter functional.
Let us note that the convolution kernel $h \mapsto h^{-(N+1)}$ in the definition of $NL_\varepsilon$ in \eqref{NL_DEFINITION}
does not satisfy the stated boundedness assumption \eqref{kernel_bd_hp} and it can be shown that
the cell formula would give the value $+ \infty$. This is the reason to use another scaling in the  definition of $NL_\varepsilon$ (and $F_\varepsilon$).

The main result of this paper is stated in Theorem
\ref{thm:Gammalim}, and its proof mostly relies on the techniques developed in \cite{h}.
It shows that carrying out the $\Gamma$-limit of (an extension of) $F_\varepsilon$ with respect to the strong 
$L^1$ convergence of the phase fields and the weak$\ast$-convergence of the surfactant measures, one
obtains a limit functional which is finite for $u\in BV(\Omega, \{ \alpha,\beta \})$ and 
$\mu\in \mathcal{M}_+(\Omega)$ (the space of positive Radon measures) where it takes the form
\begin{align*}
	F(u,\mu) := \int_{S_u} k + \left| k - \frac{d \mu_a} 
		{d \h^{N-1} \llcorner S_u} \right|  \,d\h^{N-1} + |\mu_s|(\Omega) \,.
\end{align*}
In the formula $S_u$ denotes the jump set of $u$; $\mu_a$,  $\mu_s$ denote the
absolutely continuous part and the singular part of $\mu$ respectively with respect to $\h^{N-1} \llcorner S_u$; $k$ is a constant depending on $\alpha, \beta$ and the space dimension $N$.
Hence, we see that increasing the surfactant density at the interface lowers the
energy cost of the phase boundary as long as the density does not exceed $k$,
while increasing the surfactant density further increases the energy cost, and adding surfactant
away from the interface immediately costs energy.
In this respect, our limit energy behaves as the one found as discrete-to-continuum 
$\Gamma$-limit, modeling the coarse-graining process originating from the microscopic 
Blume-Emery-Griffiths ternary surfactant model, see \cite{abcs, acs} (see also \cite{cfs1, cfs2} for their discrete-to-continuum flat flow).
Note that this behaviour differs from the one found in \cite{fms} for the local Modica-Mortola model with surfactant; i.e.
\begin{equation*}
	\frac{1}{\varepsilon} \int_\Omega W(u)  \, dx + \varepsilon \int_\Omega |\nabla u|^2 \,dx
		+ \varepsilon \int_\Omega (\rho - |\nabla u|)^2 \,dx \,,
\end{equation*}
where the limit energy only depends on the surfactant density on the interface, and moreover
is non-increasing in the surfactant density on the interface.

\section{Compactness and \(\Gamma\)-convergence}
\subsection{Notation}
We consider sets in $\R^N$ for $N \geq 2$ equipped with its standard basis $e_1, \dots, e_N$.
We write $\conv(S)$ for the convex hull of a set $S \subset \R^N$. We write $S \Delta T$ for the symmetric difference of two
 sets $S$ and $T$; sets in $\R^{N-1}$ are identified with the corresponding subsets in the hyperplane 
$\{x_N = 0\} \subset \R^N$.
For a point $x = (x_1, \dots, x_N) \in \R^N$ we write $x' := (x_1,\dots,x_{N-1})$ and also $x = (x', x_N)$.
We write $Q_R(x) := x + [-R/2, R/2]^N$ for the $N$-cube with side length $R$ centered at $x \in \R^N$
and $Q_R'(x') := x' + [-R/2, R/2]^{N-1}$ for the corresponding $(N-1)$-cube.
Similarly we write $B_R(x) := \{y \in \R^N \,:\, |x-y| < R\}$ and
$B_R'(x') := \{y' \in \R^{N-1} \,:\, |x'-y'| < R\}$.
We denote by $\mathcal{M}_+(\Omega)$ the space of non-negative Radon measures on $\Omega$.
For the level sets of a function $f: A \times B\rightarrow \R$ we write
$\{ f(x,y) = a \} := \left\{ (x,y) \in A \times B \,:\, f(x,y) = a  \right\}$.
For a function of bounded variation $u \in \BV (\Omega)$ we write $S_u$ for its approximate jump set,
as defined e.g. in \cite{afp}.
We write $C$ for constants that may change their value from line to line.

\subsection{The energy functional}
We fix a bounded regular domain $\Omega \subset \R^N$ and consider the energy
$F_\varepsilon : L^1(\Omega) \times L^1(\Omega,[0,\infty)) \rightarrow [0,+ \infty]$ defined as
\begin{align*} 
	&F_\varepsilon (u_\varepsilon, \rho_\varepsilon) := \frac{1}{\varepsilon}
		\int_\Omega W(u_\varepsilon)  \, dx +
		\frac{1}{|\ln(\varepsilon)|} \int_\Omega  \int_\Omega \frac{(u_\varepsilon(y) - u_\varepsilon(x))^2}
		{|y - x|^{N+1}} \, dy \,dx \\
	&\quad + \frac{1}{|\ln(\varepsilon)|} \int_\Omega  \left| \int_\Omega 
		\frac{(u_\varepsilon(y) - u_\varepsilon(x))^2}{|y - x|^{N+1}} \, dy - 
		\rho_\varepsilon(x) \right| \,dx \,. \nonumber
\end{align*}
We extend the definition of $F_\varepsilon$ to the functional (not renamed) $F_\varepsilon: L^1(\Omega) \times \mathcal{M_+}(\Omega)\rightarrow [0,+\infty]$
by setting
\begin{equation}
	\label{energy}F_\varepsilon(u_\varepsilon,\mu_\varepsilon) := \left\{
		\begin{array}{lr}
			F_\varepsilon(u_\varepsilon, \rho_\varepsilon), & \mu_\varepsilon
				 = \rho_\varepsilon \mathcal{L}^N, \\
			+\infty, & \text{otherwise.}
		\end{array}
	\right.
\end{equation}

Moreover, we define, for a given measurable set $A \subset \Omega$, the localized functional
$F_\varepsilon(u_\varepsilon, \rho_\varepsilon, A)$ as 
\begin{align*}
	&F_\varepsilon (u_\varepsilon, \rho_\varepsilon,A) := \frac{1}{\varepsilon}
		\int_A W(u_\varepsilon)  \, dx +
		\frac{1}{|\ln(\varepsilon)|} \int_A  \int_A \frac{(u_\varepsilon(y) - u_\varepsilon(x))^2}
		{|y - x|^{N+1}} \, dy \,dx \\
	&\quad + \frac{1}{|\ln(\varepsilon)|} \int_A  \left| \int_A 
		\frac{(u_\varepsilon(y) - u_\varepsilon(x))^2}{|y - x|^{N+1}} \, dy - 
		\rho_\varepsilon(x) \right| \,dx \,,
\end{align*} 
and the extension to $L^1(\Omega) \times \mathcal{M}_+(\Omega)$ similarly as above.

We assume that the potential $W : \R \rightarrow [0, +\infty)$ satisfies for some $\alpha, \beta \in \R$,
$\alpha < \beta$
\begin{align*}
	&W(x) = 0 \, \Leftrightarrow \, x \in \{\alpha, \beta\} \\
	&W \text{ has at least linear growth at } \pm \infty\,.
\end{align*}
We define the dimensional constants
\begin{equation}\label{kappa}
	\omega_{n-1} :=\h^{N-1} \left( B_1'(0)\right) \quad \text{and}
		\quad k:= 2(\beta - \alpha)^2\omega_{n-1}
\end{equation}
and the limit energy functional $F: L^1(\Omega)\times \mathcal{M}_+(\Omega) \rightarrow [0,\infty]$ as
\begin{equation} \label{limitenergy}
	F(u,\mu) := \left\{
		\begin{array}{lr}
		\displaystyle
			\int_{S_u} k + \left| k - \frac{d \mu_a}
			{d \h^{N-1} \llcorner S_u} \right|  \,d\h^{N-1} + |\mu_s|(\Omega),
			 &u \in \BV(\Omega, \{ \alpha,\beta \}) \\
			+\infty, &\text{otherwise.}
		\end{array} \right.
\end{equation}
Here, we denoted by $\mu_a$ (resp. $\mu_s$) the absolutely continuous (resp. singular) part of $\mu$
with respect to the measure $\h^{N-1} \llcorner S_u$.
We define for a given measurable set $A \subset \Omega$ the localized functional $F(u,\mu, A)$
as for $F_\varepsilon$ above.
Recall that for a function $u \in \BV(\Omega, \{ \alpha,\beta \})$, $S_u$ can be seen
equivalently as the reduced boundary of the set $\{ u = \beta \}$.

\subsection{The main results}
This section contains the main results of the paper. In what follows, for $\varepsilon_h \to 0^+$, we state the compactness of sequences 
$(u_h),(\rho_h)$ with equibounded energy $F_{\varepsilon_h}(u_h,\rho_h)$ as well as the 
$\Gamma$-convergence result for $F_\varepsilon$ as $\varepsilon \rightarrow 0^+$. 

\begin{thm}[Compactness] \label{compactness}
Let $F_{\varepsilon_h}$ be the functional defined in (\ref{energy}). Given sequences $\varepsilon_h \rightarrow 0^+$,
$u_h \in L^1(\Omega)$ and $\rho_h  \in L^1(\Omega, [0,\infty))$
such that $\sup_h \left\| \frac{\rho_h}{\ln \varepsilon_h} \right\|_{L^1(\Omega)} < \infty$ and
$\sup_h F_{\varepsilon_h}(u_h,\rho_h) < \infty$, then, up to subsequences, 
$u_h$ converge in $L^1(\Omega)$ to some $u \in \BV(\Omega, \{\alpha,\beta\})$ and
$\frac{\rho_h}{|\ln \varepsilon_h|} \mathcal{L}^N$
converge in the weak$\ast$ sense to some $\mu \in \mathcal{M}_+(\Omega)$.
\end{thm}
\proof
The result is a direct consequence of the compactness result for $F_\varepsilon$
with $\rho_\varepsilon \equiv 0$ from \cite[Proposition 2.12]{h}
and of the weak$\ast$ compactness of $\mathcal{M}_+(\Omega)$.\qed\\

In the following theorem we states our $\Gamma$-convergence result. As it is customary in this setting 
(see \cite{b, dm}) the $\Gamma$-convergence  will be understood with respect to the product topology
given by the strong $L^1$ topology of the functions and the weak$\ast$ topology of the measures. 
The proof of the result is a consequence of Proposition \ref{liminf} and Proposition \ref{limsup}
that will be proven in the Subsections \ref{subsec:liminf} and \ref{subsec:limsup}.\\

\begin{thm}[$\Gamma$-convergence]\label{thm:Gammalim}
Let $F_\varepsilon, F$ be defined as in (\ref{energy}) resp. (\ref{limitenergy}).
Then the following $\Gamma$-convergence result holds true.
\begin{equation*}
	\Gamma\text{-}\lim\limits_{\varepsilon \rightarrow 0^+} F_{\varepsilon}(\cdot, \cdot |\ln \varepsilon|) 
		= F \,.
\end{equation*}
\end{thm}

\begin{rem}
    We observe that, with the aim of letting the surfactant term contribute to $F_\varepsilon$ as $\varepsilon\to0$ we need that $\rho_\varepsilon$ is of the same order as the gradient like term, which is of logarithmic scaling in $\varepsilon$. Hence in the $\Gamma$-convergence result (see Theorem \ref{thm:Gammalim}) we consider sequences $\mu_\varepsilon$ such that $\frac{\mu_\varepsilon}{|\ln \varepsilon|}$ converge in $\mathcal{M}_+(\Omega)$.
\end{rem}

\subsection{Preliminaries}
In this subsection, we collect preliminary results, mostly from \cite{h}, and
consider an energy defined on two disjoint sets related to
the nonlocal part of our energy functional $F_\varepsilon$:
we define for measurable disjoint sets $A,B \subset \R^N$  and $S \subset \R^N$
\begin{equation*}
G(A,B,S) := \int_{A \cap S} \int_{B \cap S} \frac{1}{|y-x|^{N+1}} \, dy \,dx 
\end{equation*}
and
\begin{equation*}
G(A,B) := G(A,B,\R^N) \,.
\end{equation*}
Let us now recall some results of \cite{h} that will be used later on.

\begin{lem}[ \cite{h}, Lemma 2.4, Estimate of $G$ for a cylinder and its complement inside a fattened hyperplane]
	 \label{cylindercomplement}
Given $d,l \in \R$ such that $0 \leq d<l \leq 8/3$ and the set $A_R = Q_R'(0) \times (d/2,l/2)$, it holds that
\begin{equation*}
	G \left(A_R, Q_R'(0)^C \times (-l/2,-d/2) \right) \leq C(N) R^{N-1} \left( 1 - \ln(l/2) \right) \,.
\end{equation*}
\end{lem}

\begin{lem} [\cite{h}, Remark 2.5, Estimate of $G$ for a cylinder and the complement
of a cone inside a fattened hyperplane]	\label{cylindercone}
Let  $B_\infty :=\R^{N-1} \times  (-l/2,-d/2)$ be a fattened hyperplane
and let $\conv S$  the convex envelope of a set $S \subset \R^N$. It holds for the energy of a cube $A_R:=Q_R'(0) \times (d/2,l/2)$ and the complement of a pyramid inside $B_\infty$ that
\begin{equation*}
	G\left( A_R, \left(\conv\left( Q_R'(0) \cup \left\{(-l/2,0)\right\} \right)\right)^C \cap B_\infty \right) \leq 
		C(N) R^{N-1} \left( 1 - \ln(l/2) \right) \,.
\end{equation*}
\end{lem}

\begin{cor}[\cite{h}, Corollary 2.8, Lower bound of $G$ on a special cylinder]\label{correction_cylinderbound2}
Given $r >0$, a cylinder $Q_R'(0) \times (-L/2,L/2)$, a number $\lambda > 0$, a constant $c>0$
and $A,B \subset \R^N$ two disjoint sets such that
\begin{align*}
	&\frac{|Q_R'(0) \times (-L/2,-L/2+r) \cap B|}{|Q_R'(0) \times (-L/2,-L/2+r)|} > 1 - \lambda^2 \,, \\
	&\frac{|Q_R'(0) \times (L/2-r,L/2) \cap A|}{|Q_R'(0) \times (L/2-r,L/2)|} > 1 - \lambda^2 \text{ and} \\
	&|Q_R'(0) \times (-L/2,L/2) \setminus (A \cup B)| < c \varepsilon 
		\left(1-3\lambda - \frac{6}{|\ln \varepsilon|}\right)\,,
\end{align*}
then, for $r / |\ln \varepsilon| < 8/3$, it holds that
\begin{align*}
	&G \left( A,B,Q_R'(0) \times (-L/2,L/2) \right) \\
	&\quad \geq R^{N-1}\omega_{n-1} \left(1 - 3\lambda  - \frac{6}{|\ln \varepsilon|}\right)
		\left(1 - 2\xi(N-1) \left(2\lambda +  \frac{c}{|\ln \varepsilon|} \right) -  \frac{2}{|\ln \varepsilon|}\right)\\
	&\quad\quad\left( \ln\left( \frac{r}{8 |\ln \varepsilon|}\right) -
		\ln \left( \frac{\varepsilon |\ln \varepsilon|}{2} \right)
		\right) - R^{N-1} C(N) \left( 1 - \ln\left( \frac{r}{2 |\ln \varepsilon|}\right)  \right) \,,
\end{align*}
where $\xi(N-1)$ and $C(N)$ are dimensional constants.
\end {cor}

\begin{lem}[\cite{h}, Lemma 2.10, Energy upper bound for a special sequence]\label{recoverycube}
Given a cube $Q_R'(x)$, a number $l > 0$ and the family of functions
\begin{equation*}
	u_\varepsilon: Q_R'(x) \times (-l/2,l/2) \rightarrow \R \,,\, y \mapsto \left\{
	\begin{array}{lr}
		\beta, &y_N \in \left(-\frac{l}{2}, - \frac{\varepsilon}{2 |\ln \varepsilon|} \right)\\
		\beta + (\alpha - \beta) \frac{ |\ln \varepsilon|}{\varepsilon} 
			\left( y_N + \frac{\varepsilon}{2 |\ln \varepsilon|} \right), 
		& y_N \in \left( -\frac{\varepsilon}{2 |\ln \varepsilon|},
			\frac{\varepsilon}{2 |\ln \varepsilon|} \right) \\
		\alpha, &y_N \in \left( \frac{\varepsilon}{2 |\ln \varepsilon|}, \frac{l}{2} \right)
	\end{array} \right. \,,
\end{equation*}
it holds that $u_\varepsilon$ converges in $ L^1 \left( Q_R'(x) \times (-l/2,l/2) \right)$ to $u$ as
$\varepsilon \rightarrow 0^+$, where
\begin{equation*}
	 u: Q_R'(x) \times (-l/2,l/2) \rightarrow \R \,,\, y \mapsto \left\{
		\begin{array}{lr}
		\beta, &y_N \in (-l/2,0) \\
		\alpha, &y_N \in (0,l/2)
	\end{array} \right.
\end{equation*}
and
\begin{equation*}
	\limsup_{\varepsilon \rightarrow 0^+} F_\varepsilon(u_\varepsilon, 0, Q_R'(x) \times (-l/2,l/2)) \leq
		k R^{N-1}=k\mathcal{H}^{N-1}( Q'_R(x) ) \,.
\end{equation*} 
\end{lem}

The following Lemma allows us to estimate the nonlocal part of the energy for a function
$u$ constant equal to $\alpha$ on a cylinder above the plane $\{x_N = 0\}$ and having an affine
transition from $\alpha$ to $\beta$ on an affinely deformed cylinder below this hyperplane
with the energy of a function from Lemma \ref{recoverycube}. This will be useful in the proof
of the limsup inequality on polyhedral sets below. Note that the roles of $\alpha$ and $\beta$
may be interchanged without making further changes.
\begin{lem} \label{mixedinteraction}
Given a cube $Q_R'(0)$, a number $l > 0$, and a affine function $h:\R^{N-1} \rightarrow \R$
such that for all $x' \in \overline{Q_R'(0)}$  it holds that $h(x') < 0$, we consider a function
$u: \{ x \in \R^N \,:\, x' \in Q'_R(0), h(x') < x_N < l \} \rightarrow \R$, which is constant equal to
$\alpha$ on $Q_R'(0) \times (0,l)$, equals $\beta$ on the graph of $h$ and has an affine transition
between $\alpha$ and $\beta$ on the set $\{ x\in \R^N \,:\, h(x') < x_N < 0 \}$.
Let us moreover define $\hat u: \{ x \in \R^N \,:\, x' \in Q'_R(0) \textcolor{blue}{,} \ \min_{Q_R'(0)} h < x_N < l\} \rightarrow \R$,
\begin{equation*}
	\hat u(x) := \left\{
	\begin{array}{lr}
		\alpha - (\beta - \alpha) \frac{ \min_{Q_R'(0)} h}{\max_{Q_R'(0)} h} x_N
			&x_N \in (\min_{Q_R'(0)} h,0), \\
		\alpha, &x_N \in (0, l)
	\end{array} \right. \,.
\end{equation*}
Then it holds that
\begin{align*}
	F_\varepsilon \left( u, \{ x \in \R^N \,:\, x' \in Q'_R(0), h(x') < x_N < l \}, 0 \right) \leq
		& F_\varepsilon \left( \hat u, Q_R'(0) \times \left( \min_{Q_R'(0)} h, l \right), 0 \right) \\
	&\quad + \frac{1}{\varepsilon} \int_{\{ x \in \R^N \,:\, x' \in Q'_R(0), h(x') < x_N < l \}}W(u) \,dx \,.
\end{align*}
\end{lem}
\begin{proof}
We extend the function $u$ linearly to the set  $Q_R'(0) \times \left( \min_{Q_R'(0)} h, l \right)$. Then it holds
$F_\varepsilon \left( u, \{ x \in \R^N \,:\, x' \in Q'_R(0), h(x') < x_N < l \}, 0 \right) \leq
F_\varepsilon \left( u, Q_R'(0) \times \left( \min_{Q_R'(0)} h, l \right), 0 \right)$.
Since for any two fixed points $x,y \in
\{ x \in \R^N \,:\, x' \in Q'_R(0), h(x') < x_N < l \}$ it holds$|u(y) -u(x)| \leq |\hat u(y) - \hat u(x)|$, it follows
\begin{align*}
	F_\varepsilon \left( u, Q_R'(0) \times \left( \min_{Q_R'(0)} h, l \right) \right) &
		\leq F_\varepsilon \left( \hat u, Q_R'(0) \times \left( \min_{Q_R'(0)} h, l \right) \right) \\
	&\quad + \frac{1}{\varepsilon} \int_{\{ x \in \R^N \,:\, x' \in Q'_R(0), h(x') < x_N < l \}}W(u) \,dx \,.
\end{align*}
\end{proof}

\subsection{$\Gamma$-liminf inequality}\label{subsec:liminf}
\begin{prop}[$\Gamma$-liminf inequality] \label{liminf}
Given sequences $(\varepsilon_h)_h$, $(\rho_{\varepsilon_h})_h$ and $(u_{\varepsilon_h})_h$
such that $\varepsilon_h \in \R$ satisfy
$\varepsilon_h \rightarrow 0^+$, $\frac{\rho_{\varepsilon_h}}{|\ln \varepsilon_h|}$
converge weak$\ast$ to a measure $\mu \in \mathcal{M}_+(\Omega)$ and $u_{\varepsilon_h}$ converge to  a 
function $u$ in $L^1(\Omega)$ as $h \rightarrow +\infty$, it holds that
\begin{equation*}
	\liminf_{h \rightarrow \infty} F_{\varepsilon_h}(u_{\varepsilon_h}, \rho_{\varepsilon_h}) \geq F(u,\mu) \,.
\end{equation*}
\end{prop}
\begin{proof}
Without loss of generality we may assume that
$\liminf_{h \rightarrow \infty} F_{\varepsilon_h}(u_{\varepsilon_h},\rho_{\varepsilon_h}) < \infty$
and therefore by
Theorem \ref{compactness} $u \in \BV \left( \Omega, \{ \alpha,  \beta \} \right)$.
In order to simplify the notation, we write $\varepsilon$ instead of $\varepsilon_h$ and do not relabel
subsequences. 
Up to extracting a subsequence, we may assume
\begin{equation*}
	\liminf_{\varepsilon \rightarrow 0^+} F_\varepsilon(u_\varepsilon,\rho_\varepsilon)
		= 	\lim_{\varepsilon \rightarrow 0^+} F_\varepsilon(u_\varepsilon,\rho_\varepsilon) \,.
\end{equation*}

Let us set
\begin{equation*}
	g_\varepsilon(x) := \frac{1}{\varepsilon} W(u_\varepsilon)(x) +
		\frac{1}{|\ln \varepsilon|} \int_\Omega \frac{(u_\varepsilon(y) - u_\varepsilon(x))^2}
		{|y - x|^{N+1}} \, dy + \frac{1}{|\ln \varepsilon|} \left|
		\int_\Omega \frac{(u_\varepsilon(y) - u_\varepsilon(x))^2}{|y - x|^{N+1}} \,dy
		 - \rho_\varepsilon(x) \right|\,.
\end{equation*}
Then there exists a Radon measure $\gamma$ on $\Omega$ such that, again up to subsequences,
\begin{equation*}
	g_\varepsilon \mathcal{L}^N \weak* \gamma \,.
\end{equation*}
By the lower semicontinuity of the total variation with respect to weak$\ast$ convergence we get
\begin{equation*}
	\lim_{\varepsilon \rightarrow 0^+} F_\varepsilon(u_\varepsilon, \rho_\varepsilon)
		= \lim_{\varepsilon \rightarrow 0^+} \int_\Omega g_\varepsilon(x) \,dx \geq \int_\Omega \,d\gamma
		= \int_\Omega \frac{d \gamma}{d \h^{N-1} \llcorner S_u} \,d \h^{N-1}
		\llcorner S_u + \int_\Omega \,d \gamma_s,.
\end{equation*}
Hence, it is enough to show that
\begin{equation*}
	\int_\Omega \frac{d \gamma}{d \h^{N-1} \llcorner S_u} \,d \h^{N-1}\llcorner S_u 
		+ \int_\Omega \,d \gamma_s
		\geq \int_\Omega k + \left| k- \frac{d \mu_a}{d \h^{N-1} \llcorner S_u} \right| \,d\h^{N-1} \llcorner S_u
		+ \int_\Omega 1 \,d\mu_s \,.
\end{equation*}
This means that it suffices to prove the following two inequalities 
\begin{equation} \label{1)}
	\frac{d \, \gamma}{d \h^{N-1} \llcorner S_u}(x) \geq  k + \left| k- \frac{d \mu_a}{d \h^{N-1}
		\llcorner S_u} \right| \quad \text{for } \h^{N-1} \text{-a.e. } x \in S_u \,,
\end{equation}
\begin{equation} \label{2)}
	\frac{d \gamma_s}{d \mu_s} \geq 1  \quad \text{for } \mu_s \text{-a.e. } x \in \Omega\,,
\end{equation}
which imply our assertion. Let us use the ad-hoc notation $Q_R^\nu(x)$ for a cube obtained by rotating
$Q_R(x)$ in such a way that it has a face orthogonal to the vector $\nu \in \R^N$,
$Q_R^{\nu +}(x) := \left\{ y \in Q_R^\nu(x) \,:\, \nu \cdot (y-x) > 0 \right\}$
and $Q_R^{\nu -}(x) := \left\{ y \in Q_R^\nu(x) \,:\, \nu \cdot (y-x) < 0 \right\}$.
For $\h^{N-1}$-almost every $x \in S_u$ it holds that (see e.g. \cite[Theorems 2.22 and 3.59]{afp})
\begin{align*}
	&\text{it exists the measure theoretic outer normal } \nu = \nu_{S_u}(x) \,, \\
	&	\begin{array}{ll} 	\displaystyle
		\frac{1}{|Q_R^\nu (x)|}\int_{Q_R^\nu(x)} u - \left(\alpha \chi_{Q_R^{\nu +}(x)} + \beta
		 	\chi_{Q_R^{\nu -}(x)} \right) \,dy \rightarrow 0 & \text{as} \quad R \rightarrow 0^+ \,,\\[10pt]
            \displaystyle
		\frac{\h^{N-1} \llcorner S_u (Q_R^ \nu(x))}{R^{N-1}}  \rightarrow 1 & \text{as} 
			\quad R \rightarrow 0^+ \,,\\[10pt]
            \displaystyle
		\frac{\gamma \left( Q_R^ \nu(x) \right)} {\h^{N-1} \llcorner S_u (Q_R^ \nu(x))}
 			\rightarrow \frac{d \, \gamma}{d\h^{N-1} \llcorner S_u}(x) & \text{as} 
			\quad R \rightarrow 0^+, \\[10pt]
            \displaystyle
		\frac{\mu \left( Q_R^ \nu(x) \right)} {\h^{N-1} \llcorner S_u (Q_R^ \nu(x))}
 			\rightarrow \frac{d \, \mu}{d\h^{N-1} \llcorner S_u}(x) & \text{as} 
			\quad R \rightarrow 0^+ \,.
	\end{array}
\end{align*}
Moreover, there exists an at most countable set $\mathbf{B}(x)$
(the set of radii for which $\mu(\partial Q_R^\nu (x)) \not= 0$ or $\gamma(\partial Q_R^\nu (x)) \not= 0$)
such that
\begin{equation*}
	\begin{array}{lr}
    \displaystyle
		\lim_{\varepsilon \rightarrow 0^+} F_\varepsilon(u_\varepsilon, \rho_\varepsilon, Q_R^ \nu(x)) = 
			\lim_{\varepsilon \rightarrow 0^+}\int_{Q_R^ \nu(x)} g_\varepsilon \,dy =
			\gamma(Q_R^ \nu(x)) & \text{if } R \not\in \mathbf{B}(x) \text{ is sufficiently small}
	\end{array}
\end{equation*}
and
\begin{equation*}
	\begin{array}{lr}
    \displaystyle
		\lim_{\varepsilon \rightarrow 0^+} \frac{1}{|\ln \varepsilon|} 
			\int_{Q_R^ \nu(x)} \rho_\varepsilon \,dy =
			\mu(Q_R^ \nu(x)) & \text{if } R \not\in \mathbf{B}(x) \text{ is sufficiently small} \,.
	\end{array}
\end{equation*}
Therefore, given any constant $\lambda > 0$, for $\h^{N-1}$-a.e. $x \in S_u$
there exists a cube $Q_R^\nu (x) \subset \Omega$ such that
\begin{align} \label{temp5}
	&\frac{d \, \gamma}{d \h^{N-1} \llcorner S_u}(x) + \lambda \geq \frac{1}{R^{N-1}}
		\lim_{\varepsilon \rightarrow 0^+} F_\varepsilon(u_\varepsilon, Q_R^ \nu(x)) \,, \\
	& |Q_R^{\nu +}(x) \, \Delta \, \left(\{ u = \alpha\}  \cap Q_R^{\nu +}(x)\right) |
		< \frac{\lambda^2}{4} R^N \,, \nonumber\\
	& |Q_R^{\nu -}(x) \, \Delta \, \left(\{ u = \beta \}  \cap Q_R^{\nu -}(x)\right)| < \frac{\lambda^2}{4} R^N
		 \nonumber \\
	&\left| \frac{d \, \mu}{d \h^{N-1} \llcorner S_u}(x) - \frac{1}{R^{N-1}}
		\lim_{\varepsilon \rightarrow 0^+} \int_{Q^\nu_R(x)} \rho_\varepsilon\,dy \right| < \lambda \nonumber \,,
\end{align}
and therefore, by the $L^1$ convergence of $u_\varepsilon$ to $u$, for all $\varepsilon$ sufficiently small it holds
\begin{align*}
	& |Q_R^{\nu +}(x) \cap A_\varepsilon| > (1 - \lambda^2) \frac{R^N}{2} \,, \\
	& |Q_R^{\nu -}(x) \cap B_\varepsilon| > (1 - \lambda^2) \frac{R^N}{2} \,.
\end{align*}
where, given a fixed $\delta \in \left( 0, (\beta - \alpha)/2 \right)$, $A_\varepsilon$ and $B_\varepsilon$ are defined as
\begin{align*}
	&A_\varepsilon := \{ x \,:\, u_\varepsilon(x) < \alpha + \delta \} \\
	&B_\varepsilon  := \{ x \,:\, u_\varepsilon(x) > \beta - \delta \} \,.
\end{align*}
Since $F_\varepsilon(u_\varepsilon, \rho_\varepsilon)$ is uniformly bounded in $\varepsilon$, we have
for a positive constant $C$ that $|Q_R^\nu(x) \setminus (A_\varepsilon \cup B_\varepsilon)| 
< C\varepsilon$ and in particular there exists a constant $c > 0$
such that
\begin{equation}
\label{temp5.1}
	|Q_R^\nu(x) \setminus (A_\varepsilon \cup B_\varepsilon)| <
		c \varepsilon \left(1-3\lambda - \frac{6}{|\ln \varepsilon|}\right).
\end{equation} 
By the invariance of the energy $F_\varepsilon$ under rotation and translation of the domain,
from now on we may assume that $\nu = e_N$ and $x = 0$, and by (\ref{temp5.1})
we can now apply Corollary \ref{correction_cylinderbound2} for $L := R$ and $r= R/2$. It implies
\begin{align} \label{temp6}
	&\frac{1}{|\ln \varepsilon|} G \left( A,B,Q_R(0) \right) \nonumber\\
	&\quad \geq \frac{1}{|\ln \varepsilon|} R^{N-1}\omega_{n-1}
		\left(1 - 3\lambda  - \frac{6}{|\ln \varepsilon|}\right)
		\left(1 - 2\xi(N-1) \left(2\lambda +  \frac{c}{|\ln \varepsilon|} \right) -  \frac{2}{|\ln \varepsilon|}\right)
		\nonumber \\
	&\quad\quad\left( \ln\left( \frac{r}{8 |\ln \varepsilon|}\right) -
		\ln \left( \frac{\varepsilon |\ln \varepsilon|}{2} \right)
		\right) - \frac{1}{|\ln \varepsilon|} R^{N-1} C(N) 
		\left( 1 - \ln\left( \frac{r}{2 |\ln \varepsilon|}\right)  \right) \nonumber\\
	&\quad \longrightarrow R^{N-1}\omega_{n-1} (1 - 3\lambda)
		(1 - 4\xi(N-1)\lambda) \quad \text{as } \varepsilon \rightarrow 0^+ \,.
\end{align}
In order to shorten the notation, we set
\begin{equation*}
	I_\varepsilon(x) := \int_\Omega \frac{|u_\varepsilon(y) - u_\varepsilon(x)|^2}{|y-x|^{N+1}} \,dy\,.
\end{equation*}
We next distinguish the two cases \begin{equation}
\label{first_case}
    k- \frac{d \mu_a}{d \h^{N-1}	\llcorner S_u}(0) \geq 0 
\end{equation}
 or\begin{equation}
    k- \frac{d \mu_a}{d \h^{N-1}	\llcorner S_u}(0)< 0.  
 \label{second_case}\end{equation} 
 In the first case, so when \eqref{first_case} holds,
by (\ref{temp5}) and (\ref{temp6}) we get
\begin{align*}
	\frac{d \gamma}{d \h^{N-1} \llcorner S_u}&(0) + \lambda \geq \frac{1}{R^{N-1}}
		\lim_{\varepsilon \rightarrow 0^+} F_\varepsilon(u_\varepsilon,\rho_\varepsilon, Q_R(0)) \\
	&\geq \frac{1}{R^{N-1}} \liminf_{\varepsilon \rightarrow 0^+} \, \frac{1}{|\ln \varepsilon|} \Bigg(	
		\int_{Q_R(0)} I_\varepsilon(x) \,dx + \int_{Q_R(0)} |I_\varepsilon(x) - \rho_\varepsilon(x)| \,dx \Bigg) \\
	&\geq \frac{1}{R^{N-1}} \liminf_{\varepsilon \rightarrow 0^+} \, \frac{1}{|\ln \varepsilon|} \Bigg(	
		\int_{Q_R(0)} I_\varepsilon(x) \,dx + \int_{Q_R(0)} (I_\varepsilon(x) - \rho_\varepsilon(x)) \,dx \Bigg) \\		
	&\geq \frac{1}{R^{N-1}} \liminf_{\varepsilon \rightarrow 0^+} \, \frac{1}{|\ln \varepsilon|} \Bigg(
		2 \int_{A_\varepsilon \cap Q_R(0)} \int_{B_\varepsilon \cap Q_R(0)}
		\frac{|u_\varepsilon(y) - u_\varepsilon(x)|^2}{|y-x|^{N+1}} \,dy \,dx \\
	&+ 2 \int_{B_\varepsilon \cap Q_R(0)} \int_{A_\varepsilon \cap Q_R(0)}
		\frac{|u_\varepsilon(y) - u_\varepsilon(x)|^2}{|y-x|^{N+1}} \,dy \,dx
		- \int_{Q_R(0)} \rho_\varepsilon(x) \,dx \Bigg) \\
	&\geq \frac{1}{R^{N-1}} \liminf_{\varepsilon \rightarrow 0^+}  \frac{1}{|\ln \varepsilon|}
		4 (\beta - \alpha - 2 \delta)^2 G (A_\varepsilon, B_\varepsilon, Q_R(0)) \\
	&\quad - \frac{1}{R^{N-1}}\limsup_{\varepsilon \rightarrow 0^+}
		\frac{1}{|\ln \varepsilon|}\int_{Q_R(0)} \rho_\varepsilon(x) \,dx \\
	&\geq 4 (\beta - \alpha - 2 \delta)^2 \omega_{n-1} (1 - 3\lambda) (1 - 4\xi(N-1)\lambda)
		- \frac{d \, \mu}{d \h^{N-1} \llcorner S_u}(0) -\lambda \\
	&=  2 k  \frac{(\beta - \alpha - 2 \delta)^2}{(\beta - \alpha)^2} (1 - 3\lambda) (1 - 4\xi(N-1)\lambda)
		- \frac{d \, \mu}{d \h^{N-1} \llcorner S_u}(0) -\lambda
\end{align*}
Since $\lambda$ and $\delta$ are arbitrary, this implies (\ref{1)}) for the case (\ref{first_case}).

In the second case, so when \eqref{second_case} holds, we have that 
(\ref{1)}) is equivalent to
\begin{equation*}
	 \frac{d \gamma_a}{d \h^{N-1} \llcorner S_u}(0) \geq \frac{d \mu_a}{d \h^{N-1} \llcorner S_u}(0)\,.
\end{equation*}
By (\ref{temp5}) we get
\begin{align*}
	\frac{d \, \gamma}{d \h^{N-1} \llcorner S_u}&(0) + \lambda \geq \frac{1}{R^{N-1}}
		\lim_{\varepsilon \rightarrow 0^+} F_\varepsilon(u_\varepsilon,\rho_\varepsilon, Q_R(0)) \\
	&\geq \frac{1}{R^{N-1}} \liminf_{\varepsilon \rightarrow 0^+} \, \frac{1}{|\ln \varepsilon|} \Bigg(	
		\int_{Q_R(0)} I_\varepsilon(x) \,dx + \int_{Q_R(0)} |I_\varepsilon(x) - \rho_\varepsilon(x)| \,dx \Bigg) \\
	&\geq \frac{1}{R^{N-1}} \liminf_{\varepsilon \rightarrow 0^+} \, \frac{1}{|\ln \varepsilon|} \Bigg(	
		\int_{Q_R(0)} I_\varepsilon(x) \,dx + \int_{Q_R(0)} (\rho_\varepsilon(x) - I_\varepsilon(x) ) \,dx \Bigg) \\		
	&\geq \frac{1}{R^{N-1}} \liminf_{\varepsilon \rightarrow 0^+} \, \frac{1}{|\ln \varepsilon|}
		\int_{Q_R(0)} \rho_\varepsilon(x) \,dx \\
	&\geq  \frac{d \mu}{d \h^{N-1} \llcorner S_u}(0) -\lambda \,.
\end{align*}
Since $\lambda$ is arbitrary, this implies the claim in the case (\ref{second_case}).
We are left to show (\ref{2)}). For $\mu_s$-a.e. $x_0 \in \text{supp}(\mu_s)$ it holds that
\begin{equation*}
	\frac{d \gamma_s}{d \mu_s}(x_0) = \lim_{R \rightarrow 0^+} \frac{\gamma_s(Q_R(x_0))}
		{\mu_s(Q_R(x_0))}. 
\end{equation*}

Moreover, there exists an at most countable set $\mathbf{B}(x)$
(the set of radii for which $\mu(\partial Q_R(x_0)) \not= 0$ or $\gamma(\partial Q_R(x_0)) \not= 0$)
such that
\begin{equation*}
	\begin{array}{lr}
    \displaystyle
		\lim_{\varepsilon \rightarrow 0^+} F_\varepsilon(u_\varepsilon, \rho_\varepsilon, Q_R(x_0)) = 
			\lim_{\varepsilon \rightarrow 0^+}\int_{Q_R(x_0)} g_\varepsilon \,dy =
			\gamma(Q_R(x_0)) & \text{if } R \not\in \mathbf{B}(x) \text{ is sufficiently small} \,.
	\end{array}
\end{equation*}
and
\begin{equation*}
	\begin{array}{lr}
    \displaystyle
		\lim_{\varepsilon \rightarrow 0^+} \frac{1}{|\ln \varepsilon|}\int_{Q_R(x_0)} \rho_\varepsilon \,dy =
			\mu(Q_R(x_0)) & \text{if } R \not\in \mathbf{B}(x) \text{ is sufficiently small} \,.
	\end{array}
\end{equation*}

Therefore, given any $\lambda > 0$, there exists $R>0$ such that
\begin{equation*}
	\frac{d \gamma_s}{d \mu_s}(x_0) \geq \frac{\gamma_s(Q_R(x_0))}{\mu_s(Q_R(x_0))} - \lambda
\end{equation*}
and we get
\begin{align*}
\displaystyle
	\frac{d \gamma_s}{d \mu_s}(x_0) & \geq \frac{\displaystyle\lim_{\varepsilon \rightarrow 0^+}
		 F_\varepsilon(u_\varepsilon,  \rho_\varepsilon, Q_R(x_0))}{\displaystyle\lim_{\varepsilon \rightarrow 0^+}
		 \frac{1}{|\ln \varepsilon|} \int_{Q_R(x_0)} \rho_\varepsilon \,dy} - \lambda \\
	&= \lim_{\varepsilon \rightarrow 0^+} \frac{ F_\varepsilon(u_\varepsilon, \rho_\varepsilon,  Q_R(x_0))}
		{\displaystyle\frac{1}{|\ln \varepsilon|} \int_{Q_R(x_0)} \rho_\varepsilon \,dy} - \lambda \\
	&\geq \lim_{\varepsilon \rightarrow 0^+}  \displaystyle \frac{\displaystyle\frac{1}{|\ln \varepsilon|} \Big(	
		\int_{Q_R(x_0)} I_\varepsilon(x) \,dx + \int_{Q_R(x_0)} (\rho_\varepsilon(x) 
		- I_\varepsilon(x) ) \,dx \Big)}
		{\displaystyle\frac{1}{|\ln \varepsilon|}\displaystyle \int_{Q_R(x_0)} \rho_\varepsilon \,dy} - \lambda \\
	&\geq 1 - \lambda \,.
\end{align*}
By the arbitrariness of $\lambda$, this shows (\ref{2)}) and concludes the proof.
\end{proof}

\subsection{$\Gamma$-limsup inequality}\label{subsec:limsup}
We now show the limsup inequality, starting with the special case of polyhedral limit functions
and limit measures supported on polyhedral sets and a union of finitely many points.
For this we need the following
\begin{df} \label{D5.1}
\emph{Polyhedral sets} $P \subset  \R^N$ are defined as open sets with a Lipschitz boundary
that is contained in the union of finitely many affine hyperplanes in $\R^N$.
The intersection of any such hyperplane with the boundary of $P$ is called a \emph{face} of $P$. 
A \emph{polyhedral set in $\Omega$} is given as the intersection of a polyhedral 
set in $\R^N$ with $\Omega$.
Given $\alpha, \beta \in \R$, a function $u \in \BV(\Omega, \{\alpha,\beta\})$ is said to be a
\emph{polyhedral function}
corresponding to a polyhedral set $A \subset \R^N$, if $\h^{N-1} (\partial A \cap \partial \Omega) = 0$
and $\{u = \beta\} = \Omega \cap A$ (and $\{u = \alpha\} = \Omega \setminus A$).
A subset $\Sigma$ of an affine hyperplane $H$ in $\R^N$ that is a polyhedral set in $\R^{N-1}$ 
is said to be a \emph{$(N-1)$-dimensional polyhedral set} in $\R^N$.
Let $E\subset \R^N$ and let $\Sigma$ be a $(N-1)$-dimensional polyhedral set in $\R^N$. 
Denoting by $\nu_\Sigma$ the normal to the hyperplane $H$, we define the \emph{projection}
of $\Sigma$ to $E$ as
the set
\begin{equation*}
	E_\Sigma := \{ x  \in E \,:\, x \in \sigma + \nu_\Sigma \R \text{ for some } \sigma \in \Sigma
		\text{ and } (\sigma,x) \cap E = \emptyset \} \,.
\end{equation*}
Here, we wrote $ (\sigma,x) $ for the open interval with endpoints $\sigma$ and $x$.
\end{df}

\begin{figure}
\begin{tikzpicture}
	\fill [fill = black!12] (-4.5,-2.5) -- (-2.4,-2.5) -- (-2.4, 2.5) -- (-4.5,2.5) -- cycle;
	\draw (-3,2) node {$Z_{1,\varepsilon}$};
	\fill [fill = black!12] (-1.6,-2.5) -- (0.1,-2.5) -- (0.1, 2.5) -- (-1.6,2.5) -- cycle;
	\draw (-1,2) node {$Z_{2,\varepsilon}$};
	\fill [fill = black!12] (0.9,-0.2) -- (0.9,-2.5) -- (3,-2.5) --(3,1.94) -- cycle;
	\draw (2,-2) node {$Z_{3,\varepsilon}$};
	\fill [fill = black!12] (0.9,1) -- (0.9,2.5) --(2.44,2.5) -- cycle;
	\draw (1.4,2) node {$Z_{4,\varepsilon}$};
	\fill[fill = black!40] (0.9,1) --  (1.5,0.4) -- (0.9,-0.2) -- (0.9, -1) 
		--(0.1,-1) --(0.1,1)-- cycle;
	\draw (0.7,0.5) node {$L$};	
	\fill [fill = black!25] (-2.4,-0.4) -- (-1.6,-0.4) -- (-1.6, 0.4) -- (-2.4,0.4) -- cycle;
	\draw [line width = 1.2pt](-4.5,0) -- (0.5,0) -- (3,-1);
	\draw (-3,0.2) node {$F_1$};
	\draw (-1,0.2) node {$F_2$};
	\draw (1.7,-0.2) node {$F_3$};
	\draw(-2,-2.3) -- (-2,2.5);
	\draw[dashed](-1.6,-2.3) -- (-1.6,2.5);
	\draw[dashed](-2.4,-2.3) -- (-2.4,2.5);
	\draw[<->] (-1.6,-2.4) -- (-2.4,-2.4);
	\draw(-2,-2.7) node{$\frac{\varepsilon}{|\ln \varepsilon|}$};
	\draw(0.5,-2.5) -- (0.5,2.5);
	\draw[dashed](0.9,-2.5) -- (0.9,-0.2);
	\draw[dashed](0.1,-2.5) -- (0.1,2.5);
	\draw (0.5,0) -- (3,2.5);
	\draw[dashed] (0.9,1) -- (2.44,2.5);
	\draw[dashed] (0.9,-0.2) -- (3,1.94);
	\draw[dashed](0.9,1) -- (0.9,2.5);
\end{tikzpicture}
\caption{Sketch of the situation of Proposition \ref{limsupPolyhedral}; the gray region between
$Z_{1,\varepsilon}$ and $Z_{2,\varepsilon}$ is the set not belonging to $L$
where $v_\varepsilon$ is not constant.}  \label{definitionVepsilonSurf}
\end{figure}
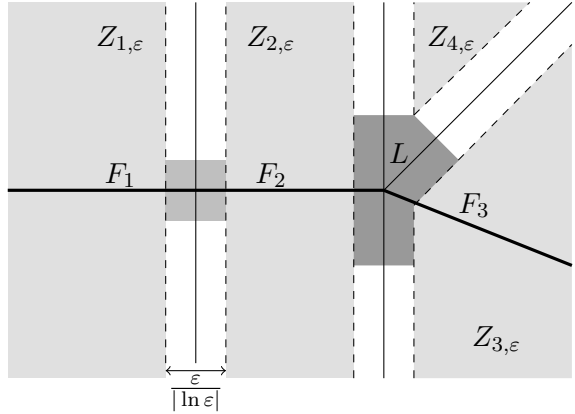

\begin{prop} \label{limsupPolyhedral}
Let $u \in \BV(\Omega, \{\alpha,\beta\})$ be a polyhedral function corresponding to a polyhedral set $P$
as in Definition \ref{D5.1}. For $n\in\mathbb{N}$ and $i\in\{1,2,\dots,n\}$ let $x_i \in \Omega \setminus S_u$,
$\zeta_i \in (0,\infty)$,
let $g$ be a piecewise constant function on polyhedral subsets of the faces 
of $S_u$ and let $\mu = \sum_{i = 1}^n \zeta_i \delta_{x_i} + g
\h^{N-1} \llcorner S_u$. Given $F_\varepsilon$ and $F$ as in  (\ref{energy}) and (\ref{limitenergy}),
there exist functions
$u_\varepsilon \in {L^1}(\Omega) $ and  $\rho_\varepsilon \in {L^1}(\Omega,[0,\infty))$, 
such that  $u_\varepsilon \rightarrow u$ in $L^1(\Omega)$
and $\frac{\rho_\varepsilon}{|\ln \varepsilon|} \mathcal{L}^N \weak* \mu$  as $\varepsilon \rightarrow 0^+$ which satisfy
\begin{equation*}
	\limsup _{\varepsilon \rightarrow 0^+} F_{\varepsilon}(u_\varepsilon, \rho_\varepsilon) 
	\leq F(u,\mu) \,.
\end{equation*}
\end{prop}

\begin{proof}
Let us assume that $u$ is of the form
\begin{equation*}
	u = \left( \beta \chi_P+ \alpha \chi_{P^C} \right) |_\Omega \,, 
		\quad \text{where } P \subset \R^N \text{ is a polyhedral set and } \h^{N-1}(\partial \Omega \cap
		\partial P) = 0
\end{equation*}
and without loss of generality (the general case follows with the same arguments) that 
\begin{equation*}
	\mu = \zeta \delta_{x_1} + \gamma_1 \h^{N-1} \llcorner F_1 \,,
\end{equation*}
where $F_1$ is a polyhedral subset of a face $F$ of $P$, i.e., $g$ is given as $g = \gamma_1 \chi_{F_1}$.

We start by defining recovery sequences for restrictions $u|_Z$ of $u$ to certain sets $Z$.\\
\noindent If it holds that
  \begin{equation} \label{facenonintersection}
	\text{$Z$ is a polyhedral set such that}
\end{equation}
  \begin{itemize}
      \item[-] $\h^{N-1}(\partial Z \cap \partial P) = 0$
      \item[-] $Z$ contains at most one
	$x_i \in \Omega\setminus S_u$, for $i=1,\dots, n$\,
  \end{itemize}
we define
\begin{equation*}
	u_{\varepsilon, Z} : Z \cap \Omega \rightarrow \R \,,\, x \mapsto \left\{
		\begin{array}{lr}
			\beta, & Z \subset P \\
			\alpha, &\text{otherwise}
		\end{array} \right.
\end{equation*}
and setting $r := 1/2 \, d(x_1, \partial Z)$ we moreover define
\begin{equation*}
	\rho_{\varepsilon, Z} : Z \cap \Omega \rightarrow \R \,,\, x \mapsto \left\{
		\begin{array}{lr}
			\frac{\zeta}{|x - x_1|^N}, & |x - x_1| \in (\varepsilon, r)
				\text{ and } x_1 \in Z \\
			0, &\text{otherwise}
		\end{array} \right. \,.
\end{equation*}
Then it holds that $u_{\varepsilon,Z} \rightarrow u$ in $L^1(Z)$ and if $x_1 \in Z$,
$\frac{\rho_\varepsilon}{|\ln \varepsilon|}\mathcal{L}^N \weak* \zeta \delta_{x_1}$,
otherwise
$\frac{\rho_\varepsilon}{|\ln \varepsilon|}\mathcal{L}^N \weak*0$.
This means that $\frac{\rho_\varepsilon}{|\ln \varepsilon|}\mathcal{L}^N \llcorner Z \weak* \mu \llcorner Z$.

Then we have
\begin{equation*}
	\limsup_{\varepsilon \rightarrow 0^+}
		F_{\varepsilon}(u_{\varepsilon,Z},\rho_{\varepsilon,Z},Z \cap \Omega) = \zeta \delta_{x_1}(Z)
		= |\mu_s|(Z) \leq F(u,\mu,Z) \,.
\end{equation*}

If instead it holds that 
\begin{equation} \label{faceintersection}
	\text{$Z$ is a polyhedral set such that}
\end{equation}
\begin{itemize} 
	\item[-] $Z \cap \partial P$ is a polyhedral subset $F_1$ of a face $F \subset \partial P$
	\item[-] $x_1 \not \in Z$
	\item[-] $g$ is constant on $F_1$
	\item[-] the orthogonal projection of $Z$ to the affine hyperplane containing $F$ is equal to $F_1$
\end{itemize}
we assume for notational convenience that $F \subset \{ x_N  = 0 \}$ and that the outer normal of
$P$ is $e_N$ on $F$ and set
\begin{equation*}
	u_{\varepsilon,Z} : Z \cap \Omega \rightarrow \R \,,\, x \mapsto \left\{
	\begin{array}{lr}
		\beta, &x_N \in \left(-\infty, - \frac{\varepsilon}{2 |\ln \varepsilon|} \right)\\
		\beta + (\alpha - \beta) \frac{|\ln \varepsilon|}{\varepsilon}
			\left(x_N + \frac{\varepsilon}{2 |\ln \varepsilon|} \right), 
		& x_N \in \left( -\frac{\varepsilon}{2 |\ln \varepsilon|},
			\frac{\varepsilon}{2 |\ln \varepsilon|} \right) \\
		\alpha, & x_N \in \left( \frac{\varepsilon}{2 |\ln \varepsilon|}, +\infty \right)
	\end{array} \right. \,.
\end{equation*}
We observe that $u_\varepsilon$ converges to $u$ in $L^1(Z)$ as $\varepsilon \rightarrow 0^+$.
We define
\begin{equation*}
		Z_\varepsilon := \{ x \in Z \,:\, d(x,\partial Z) > \varepsilon/(2 | \ln \varepsilon|) \} \,, \\
\end{equation*}
and, recalling the definition of $k$ in (\ref{kappa}), we define
\begin{equation*}
	\rho_{\varepsilon,Z_\varepsilon} : Z \cap \Omega \rightarrow \R \,,\, x \mapsto \left\{
	\begin{array}{lr}
		\displaystyle
		\frac{\gamma_1}{k} \int_{Z_\varepsilon} \frac{(u_{\varepsilon,Z}(y) - u_{\varepsilon,Z}(x))^2}
			{|y-x|^{N+1}}	\,dy, & x \in Z_\varepsilon \\
		0, & \text{otherwise}
	\end{array} \right. \,.
\end{equation*}

We now claim that $\rho_{\varepsilon,Z_\varepsilon} \weak* \mu \llcorner Z$.
Given $R > 0$, we find finitely many pairwise disjoint $R$-cubes $Q_R(x_0)$, where $x_0 \in R \mathbb{Z}^N$,
which cover $Z$ and have nonempty intersection with $Z$ up to sets of zero measure. We call the set of these cubes $\mathbf{Q}_R$.
We distinguish three cases. 
\newline
\textbf{Case 1.} Let us first fix a cube $Q_R(x_0) \in \mathbf{Q}_R$
such that $x_0 \not \in F_1$. Note that in this case $Q_{R}(x_0)\cap F_1=\emptyset$.
We claim that $\rho_{\varepsilon, Z_\varepsilon}(x) \leq C < \infty$  for all $x \in Q_R(x_0)$.
Here, $C$ does not depend on $x$. We compute (for sufficiently small $\varepsilon$)
\begin{align*}
	\rho_{\varepsilon, Z_\varepsilon}(x) &= \frac{\gamma_1}{k} \Bigg(
		\int_{\{ y \in  Z_\varepsilon \,:\, |y_N| > \frac{\varepsilon}{2 |\ln \varepsilon|}\}}
			\frac{(u_{\varepsilon,Z}(y) - u_{\varepsilon,Z}(x))^2}{|y-x|^{N+1}} \,dy,
		\\
        &\quad\quad\quad+ \int_{\{ y \in  Z_\varepsilon \,:\, |y_N| \leq \frac{\varepsilon}{2 |\ln \varepsilon|}\}}
				\frac{(u_{\varepsilon,Z}(y) - u_{\varepsilon,Z}(x))^2}{|y-x|^{N+1}} \,dy \Bigg) \\
	&\leq  \frac{\gamma_1}{k} \int_{\{ y \in  Z_\varepsilon \,:\, |y_N| \leq 
		\frac{\varepsilon}{2 |\ln \varepsilon|}\}}
		\frac{(u_{\varepsilon,Z}(y) - u_{\varepsilon,Z}(x))^2}{|y-x|^{N+1}} \,dy \\
	&\leq \frac{\gamma_1}{k} \int_{\{ y \in  Z_\varepsilon \,:\, |y_N| \leq 
		\frac{\varepsilon}{2 |\ln \varepsilon|}\}}
		\frac{(\beta - \alpha)^2}{(R/2)^{N+1}} \,dy \leq C \,.
\end{align*}
\textbf{Case 2.} 
We fix a cube $Q := Q_R(x_0) \subset Z_\varepsilon$ such that
$x_0 \in F_1$. We set $Q' := Q'_R(x_0)$ and $Z_{Q'} := (Q' \times \R) \cap Z_\varepsilon$. Then it holds that
\begin{align*}
	\frac{1}{|\ln \varepsilon|}\int_Q \rho_{\varepsilon, Z_\varepsilon}(x) \,dx &= \frac{1}{|\ln \varepsilon|}
		\Bigg( \int_Q \frac{\gamma_1}{k} \int_{Z_\varepsilon}
			\frac{(u_{\varepsilon,Z}(y) - u_{\varepsilon,Z}(x))^2}{|y-x|^{N+1}} \,dy \,dx \Bigg) \\
	&= \frac{1}{|\ln \varepsilon|} \Bigg( \int_Q\frac{\gamma_1}{k} \int_{Z_{Q'}}
			\frac{(u_{\varepsilon,Z}(y) - u_{\varepsilon,Z}(x))^2}{|y-x|^{N+1}} \,dy \,dx \\
            &\quad\quad+
            \int_Q\frac{\gamma_1}{k} \int_{Z_\varepsilon \setminus Z_{Q'}}
			\frac{(u_{\varepsilon,Z}(y) - u_{\varepsilon,Z}(x))^2}{|y-x|^{N+1}} \,dy \,dx \Bigg) \\
	&\leq \frac{1}{|\ln \varepsilon|} \Bigg( \int_Q\frac{\gamma_1}{k} \int_{Z_{Q'} \setminus Q}
			\frac{(u_{\varepsilon,Z}(y) - u_{\varepsilon,Z}(x))^2}{|y-x|^{N+1}} \,dy \,dx \\
            &\quad\quad+
			 \int_Q\frac{\gamma_1}{k} \int_{Q}
			\frac{(u_{\varepsilon,Z}(y) - u_{\varepsilon,Z}(x))^2}{|y-x|^{N+1}} \,dy \,dx \\
	&\quad +  \int_{Z_{Q'}} \frac{\gamma_1}{k} \int_{Z_\varepsilon \setminus Z_{Q'}}
			\frac{(u_{\varepsilon,Z}(y) - u_{\varepsilon,Z}(x))^2}{|y-x|^{N+1}} \,dy \,dx 
			\Bigg)
\end{align*}
The first addend can be bounded by a constant. This follows by the same computation as in case 1 above.
The last addend, by applying Lemma \ref{cylindercomplement}, is also bounded by a constant.
We apply the liminif inequality to the secquences $u_\varepsilon=u_{\varepsilon,Z}$ and $\rho_\varepsilon := 0$ on
$\Omega := Q$ to get the following bound of the second addend:
\begin{equation*}
	\liminf_{\varepsilon \rightarrow 0^+} \frac{1}{|\ln \varepsilon|} \int_Q\frac{\gamma_1}{k} \int_{Q}
		\frac{(u_{\varepsilon,Z}(y) - u_{\varepsilon,Z}(x))^2}{|y-x|^{N+1}} \,dy \,dx
		\geq\frac{\gamma_1}{k} k \h^{N-1}(Q') = \gamma_1 \h^{N-1}(Q') \,.
\end{equation*}
Moreover, we can bound from above using Lemma \ref{recoverycube}
\begin{equation*}
	\limsup_{\varepsilon \rightarrow 0^+} \frac{1}{|\ln \varepsilon|} \int_Q\frac{\gamma_1}{k} \int_{Q}
		\frac{(u_{\varepsilon,Z}(y) - u_{\varepsilon,Z}(x))^2}{|y-x|^{N+1}} \,dy \,dx
		\leq \frac{\gamma_1}{k} k \h^{N-1}(Q') = \gamma_1 \h^{N-1}(Q') \,.
\end{equation*}
Hence, we obtain
\begin{equation}\label{rhocase2}
	\lim_{\varepsilon \rightarrow 0^+} 
		\frac{1}{|\ln \varepsilon|}\int_Q \rho_{\varepsilon, Z_\varepsilon}(x) \,dx = \gamma_1 \h^{N-1}(Q')
		= \mu_a(Q) = \mu(Q)\,.
\end{equation}
Let us note for later use that a completely analogous computation shows that		
\begin{align} \label{ucase2}
	\lim_{\varepsilon \rightarrow 0^+} &
		\frac{1}{|\ln \varepsilon|}\int_Q\int_Z \frac{(u_{\varepsilon,Z} (y) - u_{\varepsilon,Z} (x))^2}
		{|y-x|^{N+1}} \,dy \,dx
		+ \frac{1}{|\ln \varepsilon|}\int_Q \left| \int_Z 
		\frac{(u_{\varepsilon,Z} (y) - u_{\varepsilon,Z} (x))^2}{|y-x|^{N+1}} \,dy
		- \rho_{\varepsilon, Z_\varepsilon}(x) \right| \,dx \nonumber \\
	&= \h^{N-1}(Q') (k + |k - \gamma_1|) \,.
\end{align}

\textbf{Case 3.} Let us consider a cube $\hat Q \in \mathbf{Q}_R$ such that
$\hat Q \not \subset Z_\varepsilon$ and $\hat Q \cap F_1 \neq \emptyset$. Then it holds
(here we extended $u_{\varepsilon,Z}$ to $\Omega$ using the same definition as above)
\begin{align*}
	\frac{1}{|\ln \varepsilon|}&\int_{\hat Q} \rho_{\varepsilon, Z_\varepsilon}(x) \,dx 
		= \frac{1}{|\ln \varepsilon|}
		\Bigg( \int_{\hat Q \cap Z_\varepsilon} \frac{\gamma_1}{k} \int_{Z_\varepsilon}
			\frac{(u_{\varepsilon,Z}(y) - u_{\varepsilon,Z}(x))^2}{|y-x|^{N+1}} \,dy \,dx \Bigg) \\
	&= \frac{1}{|\ln \varepsilon|} \Bigg( \int_{\hat Q \cap Z_\varepsilon}\frac{\gamma_1}{k}
		\int_{Z_\varepsilon \setminus \hat Q \cap Z_\varepsilon}
			\frac{(u_{\varepsilon,Z}(y) - u_{\varepsilon,Z}(x))^2}{|y-x|^{N+1}} \,dy \,dx \\
            &\quad\quad+
			 \int_{\hat Q \cap Z_\varepsilon} \frac{\gamma_1}{k}
			 \int_{\hat Q \cap Z_\varepsilon}
			\frac{(u_{\varepsilon,Z}(y) - u_{\varepsilon,Z}(x))^2}{|y-x|^{N+1}} \,dy \,dx \Bigg) \\
	&\leq \frac{1}{|\ln \varepsilon|} \Bigg( \int_{\hat Q \cap Z}\frac{\gamma_1}{k}
		\int_{Z \setminus \hat Q \cap Z}
			\frac{(u_{\varepsilon,Z}(y) - u_{\varepsilon,Z}(x))^2}{|y-x|^{N+1}} \,dy \,dx \\
            &\quad\quad+
			 \int_{\hat Q \cap Z} \frac{\gamma_1}{k}
			 \int_{\hat Q \cap Z}
			\frac{(u_{\varepsilon,Z}(y) - u_{\varepsilon,Z}(x))^2}{|y-x|^{N+1}} \,dy \,dx \Bigg) \\
	&\leq \frac{1}{|\ln \varepsilon|} \Bigg( \int_{\hat Q}\frac{\gamma_1}{k}
		\int_{Z \setminus \hat Q \cap Z}
			\frac{(u_{\varepsilon,Z}(y) - u_{\varepsilon,Z}(x))^2}{|y-x|^{N+1}} \,dy \,dx \\
            &\quad\quad+
			 \int_{\hat Q} \frac{\gamma_1}{k}
			 \int_{\hat Q}
			\frac{(u_{\varepsilon,Z}(y) - u_{\varepsilon,Z}(x))^2}{|y-x|^{N+1}} \,dy \,dx \Bigg) \,.
\end{align*}
We note that by Lemma \ref{cylindercomplement} the first addend inside the parenthesis 
is bounded by a constant and the second one by Lemma \ref{recoverycube} is bounded by
$C |\ln \varepsilon| \h^{N-1}(\hat Q \cap F_1)$.
This means that 
\begin{equation*}
	\frac{1}{|\ln \varepsilon|}\int_{\hat Q} \rho_{\varepsilon, Z_\varepsilon}(x) \,dx \leq C
		\h^{N-1}(\hat Q \cap F_1)
\end{equation*}
for a constant independent of $\varepsilon$ and $R$. Since the number of the cubes $\hat Q$ is bounded
by $C R^{2-N}$ for $\varepsilon$ sufficiently small and $C$ independent of $\varepsilon$, it holds
\begin{equation}
\label{upper_bound_surfactant_sum_cubes}
		 \sum_{ \{ \hat Q \in \mathbf{Q}_R \,:\, \hat Q \not \subset Z_\varepsilon,
		\hat Q \cap F_1 \neq \emptyset\}} \frac{1}{|\ln \varepsilon|}
		\int_{\hat Q} \rho_{\varepsilon, Z_\varepsilon}(x) \,dx \leq C R^{N-1} R^{2-N} = C R \,.
\end{equation}
Therefore, combining \eqref{upper_bound_surfactant_sum_cubes} with  the estimate found in Case 1, we have that
\begin{align*}
	\lim_{R \rightarrow 0^+} \lim_{\varepsilon \rightarrow 0^+}& \frac{1}{|\ln \varepsilon|}
		\Bigg( \int_{\cup \{ Q \in \mathbf{Q}_R \}} \rho_{\varepsilon,Z_\varepsilon}(x) \,dx
		- \int_{\cup \{ Q \in \mathbf{Q}_R \,:\, Q \subset Z_\varepsilon, Q \cap F_1 \neq \emptyset\}}
		 \rho_{\varepsilon,Z_\varepsilon}(x) \,dx \Bigg) \\
	&=\lim_{R \rightarrow 0^+} \lim_{\varepsilon \rightarrow 0^+} \frac{1}{|\ln \varepsilon|}
		\int_{\cup \{ \hat Q \in \mathbf{Q}_R \,:\, \hat Q \not \subset Z_\varepsilon,
		\hat Q \cap F_1 \neq \emptyset\}} \rho_{\varepsilon,Z_\varepsilon}(x) \,dx\\
	&=\lim_{R \rightarrow 0^+} CR = 0
	\end{align*}
	
Hence, 
\begin{equation} \label{rhocase13}
	\lim_{R \rightarrow 0^+} \lim_{\varepsilon \rightarrow 0^+} \frac{1}{|\ln \varepsilon|}
		 \int_{\cup \{ Q \in \mathbf{Q}_R \}} \rho_{\varepsilon,Z_\varepsilon}(x) \,dx
	= \lim_{R \rightarrow 0^+} \lim_{\varepsilon \rightarrow 0^+} \frac{1}{|\ln \varepsilon|}
		\int_{\cup \{ Q \in \mathbf{Q}_R \,:\, Q \subset Z_\varepsilon, Q \cap F_1 \neq \emptyset\}}
		 \rho_{\varepsilon,Z_\varepsilon}(x) \,dx \,.
\end{equation}
Let us note for later use that a completely analogous computation shows that		
\begin{align} \label{ucase13}
	&\lim_{R \rightarrow 0^+} \lim_{\varepsilon \rightarrow 0^+}
		\frac{1}{|\ln \varepsilon|} \int_{\cup \{ Q \in \mathbf{Q}_R \}}
		\int_Z \frac{(u_{\varepsilon,Z} (y) - u_{\varepsilon,Z} (x))^2}{|y-x|^{N+1}} \,dy \,dx \\
	&\quad\quad + \frac{1}{|\ln \varepsilon|}\int_{\cup \{ Q \in \mathbf{Q}_R \}} \left|
		\int_Z \frac{(u_{\varepsilon,Z} (y) - u_{\varepsilon,Z} (x))^2}{|y-x|^{N+1}} \,dy
		- \rho_{\varepsilon, Z_\varepsilon}(x) \right| \,dx \nonumber \\
	&= \lim_{R \rightarrow 0^+} \lim_{\varepsilon \rightarrow 0^+}
		\frac{1}{|\ln \varepsilon|}
		\int_{\cup \{ Q \in \mathbf{Q}_R \,:\, Q \subset Z_\varepsilon, Q \cap F_1 \neq \emptyset\}}
		\int_Z \frac{(u_{\varepsilon,Z} (y) - u_{\varepsilon,Z} (x))^2}{|y-x|^{N+1}} \,dy \,dx \nonumber \\
	&\quad\quad + \frac{1}{|\ln \varepsilon|}
		\int_{\cup \{ Q \in \mathbf{Q}_R \,:\, Q \subset Z_\varepsilon, Q \cap F_1 \neq \emptyset\}} \left|
		\int_Z \frac{(u_{\varepsilon,Z} (y) - u_{\varepsilon,Z} (x))^2}{|y-x|^{N+1}} \,dy
		- \rho_{\varepsilon, Z_\varepsilon}(x) \right| \,dx \nonumber \,.
\end{align}
Combining (\ref{rhocase2}) and (\ref{rhocase13}) and approximating smooth functions by piecewise constant ones we obtain that for every smooth test function
$\varphi$ it holds that
\begin{equation*}
	\int_Z \varphi \rho_{\varepsilon, Z_\varepsilon} \,dx \rightarrow \int_Z \varphi \,d\mu
		\quad \text{as } \varepsilon \rightarrow 0^+\,.
\end{equation*}
This shows the claim.
We next claim that
$\limsup_{\varepsilon \rightarrow 0^+} F_\varepsilon(u_{\varepsilon,Z}, \rho_\varepsilon, Z) \leq F(u,\mu,Z)$.
Again, we extended $u_{\varepsilon,Z}$ to $\Omega$
using the same definition as above, and we note that
\begin{equation*}
	\frac{1}{\varepsilon} \int_Z W(u_{\varepsilon,Z}) \,dx \rightarrow 0
\end{equation*}
as $\varepsilon \rightarrow 0^+$. Then, by (\ref{ucase13}) and (\ref{ucase2}), we get
\begin{align*}
	\limsup_{\varepsilon \rightarrow 0^+} F_\varepsilon&(u_{\varepsilon,Z},
		\rho_{\varepsilon,Z_\varepsilon}, Z) 
		\leq	\lim_{R \rightarrow 0^+} \lim_{\varepsilon \rightarrow 0^+}
		\frac{1}{|\ln \varepsilon|} \int_{\cup \{ Q \in \mathbf{Q}_R \}}
		\int_Z \frac{(u_{\varepsilon,Z} (y) - u_{\varepsilon,Z} (x))^2}{|y-x|^{N+1}} \,dy \,dx \\
	&\quad\quad + \frac{1}{|\ln \varepsilon|}\int_{\cup \{ Q \in \mathbf{Q}_R \}} \left|
		\int_Z \frac{(u_{\varepsilon,Z} (y) - u_{\varepsilon,Z} (x))^2}{|y-x|^{N+1}} \,dy
		- \rho_{\varepsilon, Z_\varepsilon}(x) \right| \,dx \nonumber \\
	&= \lim_{R \rightarrow 0^+} \lim_{\varepsilon \rightarrow 0^+}
		\frac{1}{|\ln \varepsilon|}
		\int_{\cup \{ Q \in \mathbf{Q}_R \,:\, Q \subset Z_\varepsilon, Q \cap F_1 \neq \emptyset\}}
		\int_Z \frac{(u_{\varepsilon,Z} (y) - u_{\varepsilon,Z} (x))^2}{|y-x|^{N+1}} \,dy \,dx \\
	&\quad\quad + \frac{1}{|\ln \varepsilon|}
		\int_{\cup \{ Q \in \mathbf{Q}_R \,:\, Q \subset Z_\varepsilon, Q \cap F_1 \neq \emptyset\}} \left|
		\int_Z \frac{(u_{\varepsilon,Z} (y) - u_{\varepsilon,Z} (x))^2}{|y-x|^{N+1}} \,dy
		- \rho_{\varepsilon, Z_\varepsilon}(x) \right| \,dx \\
	&\leq  \lim_{R \rightarrow 0^+} \h^{N-1}(\cup \{ Q' \,:\, Q \in \mathbf{Q}_R \text{ and } Q \cap F_1 \neq \emptyset \})
		(k + |k - \gamma_1|) \\
	& = \h^{N-1}(F_1) (k + |k - \gamma_1|) \,.
\end{align*}
Let us now assume that there are given two disjoint sets $Z_1,Z_2$, such that both $Z_1$ and $Z_2$
satisfy (\ref{facenonintersection}) or (\ref{faceintersection}) and let us set
$u_{\varepsilon,Z_1},u_{\varepsilon,Z_2}$
and $\rho_{\varepsilon,Z_{\varepsilon,1}},\rho_{\varepsilon,Z_{\varepsilon,2}}$
for the corresponding functions as introduced above. 
We define
\begin{align*}
	&Z_{1,\varepsilon} := \{ x \in Z_1 \,:\, d(x,\partial Z_1) > \varepsilon/(2 |\ln \varepsilon|) \} \,, \\
	&Z_{2,\varepsilon} := \{ x \in Z_2 \,:\, d(x,\partial Z_2) > \varepsilon/(2 |\ln \varepsilon|) \} \,,
\end{align*}
and set $Z_\varepsilon := Z_{1,\varepsilon} \cup Z_{2,\varepsilon}$,
\begin{align*}
	&\tilde u_\varepsilon: Z_\varepsilon \cap \Omega \rightarrow \R, x \mapsto \left\{
	\begin{array}{lr}
			u_{\varepsilon,Z_1}(x), & x \in Z_{1,\varepsilon} \\
			u_{\varepsilon,Z_2}(x), &x \in Z_{2,\varepsilon}
		\end{array} \right. \,, \\ 
	&\rho_\varepsilon: (Z_1 \cup Z_2) \cap \Omega \rightarrow \R, x \mapsto \left\{
	\begin{array}{lr}
			\rho_{\varepsilon,Z_{\varepsilon,1}}(x), & x \in Z_1 \\
			\rho_{\varepsilon,Z_{\varepsilon,2}}(x), &x \in Z_2
		\end{array} \right. \,.
\end{align*}
We note that $\tilde u_\varepsilon \rightarrow u$ in $L^1$ and that
$\frac{\rho_\varepsilon}{|\ln \varepsilon|} \weak* \mu$ in any union of two sets compactly contained in $Z_1$ or $Z_2$
as $\varepsilon \rightarrow 0^+$.
Using Lemma \ref{cylindercone} we have that
\begin{align*}
	&G \left(Z_{1,\varepsilon} \cap \{\tilde u_\varepsilon = \alpha\} , 
		Z_{2,\varepsilon} \cap \{\tilde u_\varepsilon > \alpha\} , \Omega \right) < C \,, \\
	&G \left(  Z_{1,\varepsilon} \cap \{\tilde u_\varepsilon = \beta\} ,
		 Z_{2,\varepsilon} \cap \{\tilde u_\varepsilon < \beta\}, \Omega \right) < C \,.
\end{align*}
Moreover, by the very definition of $G$ it holds that
\begin{align*}
	&G \left(  Z_{1,\varepsilon} \cap \{\tilde u_\varepsilon = \alpha\} ,
		  Z_{2,\varepsilon} \cap \{\tilde u_\varepsilon = \alpha\}, \Omega \right) = 0 \,, \\
	&G \left( Z_{1,\varepsilon} \cap \{\tilde u_\varepsilon = \beta\} , 
		 Z_{2,\varepsilon} \cap \{\tilde u_\varepsilon = \beta\}, \Omega \right) = 0 \,,
\end{align*}
and the same holds true with the roles of $Z_{1,\varepsilon},Z_{2,\varepsilon}$ interchanged.
Therefore we have that
\begin{align*}
	\limsup_{\varepsilon \rightarrow 0^+} F_\varepsilon(\tilde u_\varepsilon,
		\rho_\varepsilon, Z_\varepsilon)  &
		\leq \limsup_{\varepsilon \rightarrow 0^+} F_\varepsilon
		(\tilde u_\varepsilon, \rho_\varepsilon,  Z_{1,\varepsilon})
		+ \limsup_{\varepsilon \rightarrow 0^+} F_\varepsilon
		(\tilde u_\varepsilon, \rho_\varepsilon, Z_{2,\varepsilon}) \\
	&\leq \limsup_{\varepsilon \rightarrow 0^+} F_\varepsilon
		(u_{\varepsilon,Z_1}, \rho_{\varepsilon,Z_{\varepsilon,1}},  Z_1)
		+ \limsup_{\varepsilon \rightarrow 0^+} F_\varepsilon
		(u_{\varepsilon,Z_2}, \rho_{\varepsilon,Z_{2,\varepsilon}}, Z_2) \\
	&\leq F(u,\mu,Z_1 \cup Z_2)  \,.
\end{align*}
Since we can find a covering of $\Omega$ with finitely many pairwise disjoint sets $Z_1 \dots, Z_n$
as in (\ref{facenonintersection}) or (\ref{faceintersection}),
we can repeat the procedure now using
$Z_1, \dots, Z_n$ instead of $Z_1,Z_2$ and find
functions $\tilde u_\varepsilon: \Omega \cap \left(Z_{1,\varepsilon} \cup
\dots \cup Z_{n,\varepsilon} \right)\rightarrow \R$ and 
$\rho_\varepsilon: \Omega \cap \left(Z_1 \cup \dots \cup Z_n \right)\rightarrow \R$
such that
\begin{equation*}
	\limsup_{\varepsilon \rightarrow 0^+} F_\varepsilon(\tilde u_\varepsilon, \rho_\varepsilon,
		Z_{1,\varepsilon} \cup \dots \cup Z_{n,\varepsilon} )
		\leq F(u,\mu) \,.
\end{equation*}
On the set $\Omega \setminus \left(Z_{1,\varepsilon} \cup \dots \cup Z_{n,\varepsilon} \right)$ we can find
a function $v_\varepsilon: \Omega \setminus \left(Z_{1,\varepsilon} \cup \dots \cup Z_{n,\varepsilon} \right)
\rightarrow \R$, which is piecewise affine, its restriction $v_\varepsilon|_{\partial Z_{i,\varepsilon}}$
coincides with the restritcion $\tilde u_\varepsilon|_{\partial Z_{i,\varepsilon}}$ for all $i$, the set where it is non-constant
has measure bounded by $C \frac{\varepsilon^2}{|\ln \varepsilon|^2}$  and its gradient is bounded by
$\tilde C \frac{\varepsilon}{|\ln \varepsilon|}$  (where $C$ and $\tilde C$ do not depend on
$\varepsilon$).
This can be seen as follows (see also Figure \ref{definitionVepsilonSurf}). For any pair $(Z_i, Z_j)$ with a common face $F$,
we define for all points $x \in F$ which satisfy that
$B_{\varepsilon/|\ln \varepsilon|}(x) \subset Z_i \cup Z_j$ on the segment connecting the point
$(\partial Z_{i,\varepsilon} \cap B_{\varepsilon/ (2|\ln \varepsilon|)}(x))$ and the point
$(\partial Z_{j,\varepsilon} \cap B_{\varepsilon/ (2|\ln \varepsilon|)}(x))$
the function $v_\varepsilon$ to be the affine interpolation between 
$u_{\varepsilon}(\partial Z_{i,\varepsilon} \cap B_{\varepsilon/ (2 |\ln \varepsilon|)}(x))$
and $u_{\varepsilon}(\partial Z_{j,\varepsilon} \cap B_{\varepsilon/ (2 |\ln \varepsilon|)}(x))$.
This means that the function $v_\varepsilon$ is well defined  for any point in $\Omega \setminus \left(Z_{1,\varepsilon} \cup \dots \cup Z_{n,\varepsilon} \right)$
which are in the situation as described above.
The remaining set $L$ where $v_\varepsilon$ needs to be defined
is contained in a $C \frac{\varepsilon}{|\ln\varepsilon|}$-neighbourhood of the $(N-2)$-dimensional edges
of the sets $Z_1, \dots Z_n$ and therefore has Lebesgue measure bounded by
$C \frac{\varepsilon^2}{|\ln \varepsilon|^2}$. On this set,
we choose any  affine interpolation between the boundary values for $v_\varepsilon$ whose gradient on the
affine regions is bounded by $C \frac{|\ln \varepsilon|}{\varepsilon}$ and $C$ does not depend on
$\varepsilon$. Then $v_\varepsilon$ can be non-constant only on this set $L$ or on  a
$C \frac{\varepsilon}{|\ln \varepsilon|}$-neighbourhood of the 
$(N-2)$-dimensional intersections of faces of $Z_i$ and  $P$ and therefore
on a set of measure bounded by $C \frac{\varepsilon^2}{|\ln_\varepsilon|^2}$.
We can then use Lemma \ref{cylindercone} to show that
\begin{align*}
	&G \left(Z_{i,\varepsilon} \cap \{\tilde u_\varepsilon > \alpha\} , 
		\{v_\varepsilon = \alpha\} , \Omega \right) < C \,, \\
	&G \left(  Z_{i,\varepsilon} \cap \{\tilde u_\varepsilon < \beta\} ,
		 \{v_\varepsilon = \beta\}, \Omega \right) < C
\end{align*}
and by the very definition of $G$ it holds
\begin{align*}
	&G \left(Z_{i,\varepsilon} \cap \{\tilde u_\varepsilon = \alpha\} , 
		\{v_\varepsilon = \alpha\} , \Omega \right) = 0 \,, \\
	&G \left(  Z_{i,\varepsilon} \cap \{\tilde u_\varepsilon = \beta\} ,
		 \{v_\varepsilon = \beta\}, \Omega \right) = 0\,.
\end{align*}
We define
\begin{equation*}
	u_\varepsilon: \Omega \rightarrow \R, x \mapsto \left\{
	\begin{array}{lr}
			\tilde u_\varepsilon(x), & x \in \cup_{i = 1, \dots ,n}Z_{i,\varepsilon} \\
			v_\varepsilon(x), &x \in \Omega \setminus \cup_{i = 1, \dots ,n}Z_{i,\varepsilon} \,.
		\end{array} \right.
\end{equation*}
Then $u_\varepsilon$ converges to $u$ as $\varepsilon \rightarrow 0^+$ on $\Omega$.
Recall that $\rho_\varepsilon$ is defined (piecewise) on $\Omega = (Z_1 \cup \dots \cup Z_n) \cap \Omega$
and converges weak$\ast$ to $\mu$. \newline
We now apply Lemma \ref{mixedinteraction} to see that 
$\limsup_{\varepsilon \rightarrow 0^+}  \int_{\left\{ \alpha < v_\varepsilon < \beta \right\}} \int_\Omega 
		\frac{\left(u_\varepsilon(y) - u_\varepsilon(x) \right)^2}{|y-x|^{N+1}} \,dy \,dx$
can be estimated from above using Lemma \ref{recoverycube} and Lemma \ref{cylindercomplement}
and we get
\begin{equation*}
\limsup_{\varepsilon \rightarrow 0^+} \frac{1}{|\ln \varepsilon|}  \int_{\left\{ \alpha < v_\varepsilon < \beta \right\}} \int_\Omega 
		\frac{\left(u_\varepsilon(y) - u_\varepsilon(x) \right)^2}{|y-x|^{N+1}} \,dy \,dx = 0 \,.
\end{equation*}
Hence, we obtain that
\begin{equation*}
	\limsup_{\varepsilon \rightarrow 0^+} F_\varepsilon(u_\varepsilon, \rho_\varepsilon) \leq F(u,\mu) \,.
\end{equation*}
\end{proof}

Making use of the previous proposition, we now consider the case of a
function $u\in\BV(\Omega,\{\alpha,\beta\})$ and of 
a measure $\mu$ which concentrates its $\h^{N-1}$-absolutely continuous part on projections of finitely many
$(N-1)$ -dimensional polyhedral sets on the jump set of $u$.
\begin{prop} \label{temp8}
Let $u \in \BV(\Omega, \{\alpha,\beta\})$, $n\in\mathbb{N}$ and for $i\in\{1,2,\dots,n\}$ let $x_i \in \Omega \setminus S_u$,
$\zeta_i \in (0,\infty)$,
let $g$ be a piecewise constant function on the projections on $S_u$ of finitely many $(N-1)$-dimensional
polyhedral sets in $\R^N$ and let $\mu = \sum_{i = 1}^n \zeta_i \delta_{x_i} + g
\h^{N-1} \llcorner S_u$. Then there exist functions
$u_\varepsilon \in {L^1}(\Omega) $ and  $\rho_\varepsilon \in {L^1}(\Omega,[0,\infty))$, 
such that  $u_\varepsilon \rightarrow u$ in $L^1(\Omega)$
and $\frac{\rho_\varepsilon}{|\ln \varepsilon|} \mathcal{L}^N \weak* \mu$  as $\varepsilon \rightarrow 0^+$ which satisfy
\begin{equation}
	\limsup _{\varepsilon \rightarrow 0^+} F_{\varepsilon}(u_\varepsilon, \rho_\varepsilon) 
	\leq F(u,\mu).
\end{equation}
\end{prop}

\begin{proof}
Without loss of generality, let us assume that $g = \gamma \chi_{\Sigma}$ for $\gamma \in (0,\infty)$
and there exists $\Sigma_g \subset \R^N$, a finite union of 
$(N-1)$-dimensional polyhedral subsets, such that $\Sigma$ is the projection of $\Sigma_g$ on $S_u$.
Let us moreover assume without loss of generality that 
\begin{equation*}
\mu =\zeta \delta_{x_1} + \gamma \h^{N-1} \llcorner \Sigma \,.
\end{equation*}

Then, there exists (see \cite{ab}, Section 5.4) a sequence of polyhedral functions $u_h \in \BV(\Omega, \{\alpha,\beta\})$ 
such that
\begin{equation*}
	u_h \rightarrow u \text{ in } L^1(\Omega), \quad |D u_h| \rightarrow |D u|.
\end{equation*}
We may moreover assume that $x_1 \not \in S_{u_h}$ for all $h$.
If we write $\Sigma_h$ for the projection of $\Sigma_g$ to $S_{u_h}$, we have that
$\Sigma_h$ is a polyhedral subset of $S_{u_h}$ and that
\begin{equation*}
	\mu_h := \zeta \delta_{x_1} + \gamma \h^{N-1} \llcorner \Sigma_h \weak* \mu.
\end{equation*}
We apply  Proposition \ref{limsupPolyhedral} to $u_h$ and $\mu_h$ and
use the lower semicontinuity of $\gammalimsup$  to obtain
\begin{align*}
	\gammalimsup_{\varepsilon \rightarrow 0^+}& F_{\varepsilon} (u,\mu) 
		\leq \liminf_{h \rightarrow \infty} \gammalimsup_{\varepsilon \rightarrow 0^+} 
		F_{\varepsilon} (u_h,\mu_h) \\
	& \leq \liminf_{h \rightarrow \infty} \int_{S_{u_h}} k + \left| k -  
		\frac{d (\mu_h)_a}{d \h^{N-1} \llcorner S_{u_h}} \right| \,d \h^{N-1} + |(\mu_h)_s|(\Omega) \\
	&=F(u,\mu) \,.
\end{align*}
\end{proof}

\begin{figure}
\begin{tikzpicture}
	\begin{scope}
		\clip(0,0) ellipse [x radius = 3, y radius = 2.2];
		\fill [fill = black!12] (3,-2.5) -- (-1,-3) -- (-4,1.5) -- (-1,3.5) -- cycle ;
	\end{scope}
	\begin{scope}
		\clip (-3,0) -- (0,2) -- (2,-1) -- (-1.66,-2) -- cycle;
		\draw [line width = 1.2 pt] (-3,0) .. controls (-2,-1) and (-1,2) .. (0.5,0);
		\draw [line width = 1.2 pt] (0.5,0) .. controls (1.8,-2) and (2,2) .. (2.7,-1);
	\draw [line width = 1.2 pt] (-3,0) -- (-0.5,1) -- (1,0) -- (2,0.6) -- (2.9,-0.5);
	\end{scope}
	\draw(0,0) ellipse [x radius = 3, y radius = 2.2];
	\draw (-3,0) .. controls (-2,-1) and (-1,2) .. (0.5,0);
	\draw (0.5,0) .. controls (1.8,-2) and (2,2) .. (2.7,-1);
	\draw (-3,0) -- (0,2);
	\draw (-3,0) -- (-0.5,1) -- (1,0) -- (2,0.6) -- (2.9,-0.5);
	\draw (0.4,1.6) node {$\Omega$};
	\draw (1.7,-0.7) node {$S_u$};
	\draw (-2.2,-0.4) node {$\Sigma$};
	\draw (-1.5,1.3) node {$\Sigma_g$};
	\draw (0.5,0.7) node {$\Sigma_h$};
	\draw (1.7,0.9) node {$S_{u_h}$};
\end{tikzpicture}
\caption{Sketch of the situation of Proposition \ref{temp8} in the case $N=2$. The thick black lines represents the projection of $\Sigma_g$ on the interfaces.}
\end{figure}
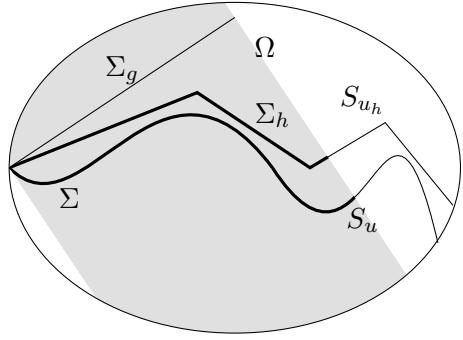

\begin{prop}[$\Gamma$-limsup inequality] \label{limsup}
Given $u \in \BV(\Omega, \{\alpha,\beta\})$ and a Radon measure $\mu \in \mathcal{M}_+(\Omega)$,
there are functions $u_\varepsilon \in {L^1}(\Omega) $, $u_\varepsilon \rightarrow u$ and  
$\rho_\varepsilon \in {L^1}(\Omega,[0,\infty))$, $\frac{\rho_\varepsilon}{|\ln \varepsilon|} \weak* \mu$
for $\varepsilon \rightarrow 0^+$, such that
\begin{equation}
	\limsup _{\varepsilon \rightarrow 0^+} F_{\varepsilon}(u_\varepsilon, \rho_\varepsilon,\Omega) 
	\leq F(u,\nu)\,.
\end{equation}
\end{prop}
\begin{proof}
Let us set
\begin{equation*}
g := \frac{d \mu}{d \h^{N-1} \llcorner S_u} \in L^1(S_u,\h^{N-1}).
\end{equation*}
We observe that there exist simple functions
\begin{equation*}
	g_h = \sum_{i=0}^{n_h} \gamma_{h,i} \chi_{K'_{h,i}} \,,
\end{equation*}
where $ K'_{h,i} \subset S_u, i = 1, \dots , n_h$, are projections of polyhedral sets  on $S_u$ and $ K'_{h,i}$
are pairwise disjoint,
which satisfy
\begin{equation*}
	\lim_{h \rightarrow \infty} g_h = g \text{ in } L^1(S_u,{\mathcal H}^{N-1}) \,.
\end{equation*}
By the boundedness of $\mu$, we find measures $\sum_{i=0}^{m_h} \zeta_{h,i} \delta_{x_{h,i}}$
supported outside $S_u$ such that
\begin{equation*}
	\sum_{i=0}^{m_h} \zeta_{h,i} \delta_{x_{h,i}} \weak* \mu  - g {\mathcal H}^{N-1} \llcorner S_u.
\end{equation*}
It follows that
\begin{equation*}
	\mu_h := g_h {\mathcal H}^{N-1} \llcorner S_u
	+ \sum_{i=0}^{m_h} \zeta_{h,i} \delta_{x_{h,i}} \weak* \mu.
\end{equation*}
We can finally combine the lower-semicontinutiy of the $\Gamma$-$\limsup$ with the result of
Proposition \ref{temp8} to obtain that
\begin{align*}
	\gammalimsup_{\varepsilon \rightarrow 0^+} F_{\varepsilon}(u,\mu)
		&\leq \liminf_{h \rightarrow \infty} \,  \gammalimsup_{\varepsilon \rightarrow 0^+}
		 F_{\varepsilon}(u,\mu_h) \\
	& \leq \liminf_{h \rightarrow \infty} \int_{S_u} k + \left| k -  
		\frac{d (\mu_h)_a}{d \h^{N-1} \llcorner S_{u_h}} \right| \,d \h^{N-1} + |(\mu_s)_h|(\Omega) \\
	&= F(u,\mu) \,.
\end{align*}
This concludes the proof.
\end{proof}

\textbf{Acknowledgments.} The first author is member of INdAM - GNAMPA Project, codice CUP E53C25002010001.


\begin{thebibliography}{a}
\bibitem{ab} Alberti, G. and Bellettini G., A non-local anisotropic
model for phase transitions: asymptotic behaviour of rescaled energies, European Journal of Applied Mathematics vol. 9, 1998.
\bibitem{ab2} Alberti, G. and Bellettini G., A non-local anisotropic
model for phase transitions. {I}. {T}he optimal profile problem, Math. Ann., 310, (1998), 527--560.
\bibitem{abs} Alberti, G., Bouchitt\'e, G. and Seppecher, P., Un r\'esultat de perturbations
	singuli\`eres avec la norme $H^{1/2}$, C.R. Acad. Sci. Paris, t. 319, S\'erie I, p. 333 --338, 1994.
 \bibitem{abcs} Alicandro, R., Braides A., Cicalese, M. and Solci, M.,
	 Discrete variational problems with interfaces, Cambridge Monographs on
	 Applied and Computational Mathematics, 40, Cambridge University Press, Cambridge, 2024.
\bibitem{acs} Alicandro, R., Cicalese, M. and Sigalotti, L., Phase transitions in presence of surfactants:
	from discrete to continuum, Interfaces Free Bound., 14 no. 1, (2012), 65--103.
\bibitem{afp} Ambrosio, L., Fusco, N. and Pallara, D., Functions of
	Bounded Variation and Free Discontinuity Problems, Clarendon Press, Oxford, 2000.
\bibitem{b} Braides, A., {$\Gamma$}-convergence for beginners, 
	Oxford Lecture Series in Mathematics and its Applications, {\bf 22},
	Oxford University Press, Oxford, 2002, xii+218.
\bibitem{cfs1} Cicalese, M., Fusco, G. and Savaré, G., Crystalline motions of discrete interfaces in the Blume-Emery-Griffiths model, Arxiv https://arxiv.org/abs/2510.18403, 2025
\bibitem{cfs2} Cicalese, M., Fusco, G. and Savaré, G., Crystalline motions of discrete interfaces in the Blume-Emery-Griffiths model: partial wetting, Arxiv https://arxiv.org/abs/2512.22870, 2025
\bibitem{ch1} Cicalese, M., Heilmann, T., Surfactants in a non-local model for phase transitions, ESAIM: Control, Optimisation and Calculus of Variations 31 (2025), 58.
\bibitem{ch2} Cicalese, M., Heilmann, T., Surfactants in the two gradient theory of phase transitions[J]. Mathematics in Engineering, 2025, 7(4), 522-552.
\bibitem{dm} Dal Maso, G., An introduction to $\Gamma$-convergence, Progress in Nonlinear Differential Equations and their Application,
	8, Birkhauserr Boston, xiv+340, 1993.
\bibitem{fms} Fonseca I., Morini M. and Slastikov, V., Surfactants in Foam Stability:
	A Phase-Field Model, Arch. Rational Mech. Anal.
\bibitem{h} Heilmann, T., $\Gamma$-convergence for nonlocal phase transitions involving the
    $H^{1/2}$ norm, preprint, https://arxiv.org/abs/2507.11054.
\bibitem{m} Modica L., The Gradient Theory of Phase Transitions and the Minimal Interface Criterion, Arch. 
	Rational Mech. Anal. 98, p. 123--142, 1987.
\bibitem{mm} Modica, L. and Mortola, S., Un esempio di $\Gamma$-convergenza, Boll. Un. Mat. Ital. B (5),
	14, p. 285--299, 1977.
\bibitem{sv1} Savin, O. and Valdinoci, E., $\Gamma$-convergence for nonlocal phase transitions,
	Ana. I. H. Poincar\'e AN 29 (2012) 479 -- 500.
\bibitem{sv2} Savin, O. and Valdinoci, E., Density estimates for a variational model driven 
	by the Gagliardo norm, J. Math. Pures Appl. 101, p. 1--26, 2014.
\bibitem{psv} Palatucci, G., Savin, O. and Valdinoci, E., Local and global minimizers for a variational energy
	involving a fractional norm, Annali di Matematica 192, p. 673--718, 2013.

	
\end{thebibliography}
\end{document}